\pdfoutput=1
\RequirePackage{ifpdf}
\ifpdf 
\documentclass[pdftex]{sigma}
\else
\documentclass{sigma}
\fi

\usepackage{tikz}

\numberwithin{equation}{section}

\newtheorem{Theorem}{Theorem}[section]
\newtheorem*{Theorem*}{Theorem}
\newtheorem{Corollary}[Theorem]{Corollary}
\newtheorem{Lemma}[Theorem]{Lemma}
\newtheorem{Proposition}[Theorem]{Proposition}
 { \theoremstyle{definition}
\newtheorem{Definition}[Theorem]{Definition}

\newtheorem{Remark}[Theorem]{Remark} }

\begin{document}
\allowdisplaybreaks

\newcommand{\arXivNumber}{2210.06024}

\renewcommand{\PaperNumber}{072}

\FirstPageHeading

\ShortArticleName{Multiplicative Characters and Gaussian Fluctuation Limits}

\ArticleName{Multiplicative Characters\\ and Gaussian Fluctuation Limits}

\Author{Ryosuke SATO}
\AuthorNameForHeading{R.~Sato}

\Address{Department of Physics, Faculty of Science and Engineering, Chuo University,\\ Kasuga, Bunkyo-ku, Tokyo 112-8551, Japan}
\Email{\href{mailto:r.sato@phys.chuo-u.ac.jps}{r.sato@phys.chuo-u.ac.jp}}

\ArticleDates{Received May 09, 2023, in final form September 19, 2023; Published online October 03, 2023}

\Abstract{It is known that extreme characters of several inductive limits of compact groups exhibit multiplicativity in a certain sense. In the paper, we formulate such multiplicativity for inductive limit quantum groups and provide explicit examples of multiplicative characters in the case of quantum unitary groups. Furthermore, we show a~Gaussian fluctuation limit theorem using this concept of multiplicativity.}

\Keywords{asymptotic representation theory; quantum groups; inductive limits; quasi-local algebras}

\Classification{17B37; 22D25; 22E66; 46L53; 46L67; 60F05}

\section{Introduction}
Asymptotic representation theory initiated by Vershik and Kerov is a unitary representation theory of inductive limits of compact groups and a significant framework to study characters of such groups (see \cite{VK81-2,VK81,VK82} etc.). For instance, using Vershik--Kerov's idea, we can describe complete lists of the extreme characters of various inductive limit groups. In particular, the characters of the infinite symmetric group ${\rm S}(\infty)$ and the infinite-dimensional unitary group~${\rm U}(\infty)$ are well studied and have intimate connections with many branches of mathematics (see \cite{BO05, BO12, BO12-2, BO13, Boyer83, Kerov:book, KOV04, Petrov14} etc.). One of the features of extreme characters of ${\rm S}(\infty)$,~${\rm U}(\infty)$ is that they are \emph{multiplicative} (see Section \ref{sect:mult}). Voiculescu \cite{Voiculescu74} first noticed this fact for ${\rm U}(\infty)$ based on the theory of type I\hspace{-0.1em}I${}_1$ factors. After that, Vershik and Kerov showed that the extreme characters of ${\rm S}(\infty)$ are also multiplicative. In Vershik--Kerov's work \cite{KV80,VK85,VK87,KV90}, the dimension group (i.e., $K_0$-group) of the group $C^*$-algebra $C^*({\rm S}(\infty))$ plays an important role. More precisely, there exists a bijection between the characters of ${\rm S}(\infty)$ and the \emph{states} of the dimension group, and the states corresponding to extreme characters are multiplicative, where the dimension group admits a natural ring structure. Moreover, this fact implies that the extreme characters of~${\rm S}(\infty)$ are multiplicative. However, inductive limits of compact groups are not generally locally compact, and hence the meaning of their group $C^*$-algebra are unclear. Therefore, one can not directly extend Vershik--Kerov's idea to general settings. Nevertheless, Boyer \cite{Boyer87} introduced an appropriate ring for an inductive system of compact groups. For ${\rm U}(\infty)$, Olshanski \cite{O16} studied such a ring, called the \emph{representation ring}, based on the theory of symmetric functions. The representation ring plays the same role as the dimension group for ${\rm S}(\infty)$, and one can show that the extreme characters of ${\rm U}(\infty)$ are multiplicative by Vershik--Kerov's approach.

The author \cite{Sato1,Sato3} initiated the asymptotic representation theory for quantum groups. Our formulation of the theory, among other things, fits Gorin's study of the $q$-deformed Gelfand--Tsetlin graph in \cite{G12}. Thus, based on Gorin's work, one can obtain a complete list of the extreme \emph{quantized} characters of the infinite-dimensional \emph{quantum} unitary group ${\rm U}_q(\infty)$, where ${\rm U}_q(\infty)$ is the inductive limit of the quantum unitary groups ${\rm U}_q(N)$.

In the paper, we will proceed with investigating quantized characters of ${\rm U}_q(\infty)$ and particularly study their multiplicativity (in the sense of Definition \ref{def:multiplicative}). Recently, Ueda \cite{Ueda22,Ueda20} extended the notion of the representation ring to more general settings. Thus, based on his idea, one can study the multiplicativity of quantized characters similarly to the classical case. In particular, we give concrete examples of multiplicative quantized characters of ${\rm U}_q(\infty)$.

We will also give an application of multiplicative characters to probability theory. As mentioned above, the character theory of ${\rm U}(\infty)$ relates to various branches of mathematics. For instance, Borodin and Bufetov discovered that specific characters of ${\rm U}(\infty)$, called one-sided Plancherel characters, produce Gaussian fields (see \cite{BB12,BB14}). Their work gives a type of central limit theorem on the universal enveloping algebras of ${\rm U}(N)$. In the paper, we show a~central limit theorem on the \emph{group von Neumann algebras} of ${\rm U}(N)$ (see Corollary \ref{cor:cha}). Our result is a new type of Gaussian limit theorems in asymptotic representation theory, but our idea is essentially the same as in studies of a central limit theorem for quantum spin chains (see~\cite{D,GVV89,GVV90,GVV91,GV89,Matsui}). Similar to those work, fluctuations of non-commutative random variables are described by \emph{quasi-local algebras}, and we can obtain an explicit algebraic description for the Gaussian limits of those fluctuations using \emph{Weyl CCR algebras} and \emph{quasi-free states}.

The organization of the paper is the following: In Section \ref{sect:2}, we discuss KMS states on a quasi-local algebras with a flow. The main theorem here gives a necessary and sufficient condition that KMS states become multiplicative in the sense of Definition \ref{def:multiplicative} (see Theorem \ref{thm:main1}). If a quasi-local algebra additionally has \emph{shift translations} (see Section \ref{subsect:shift}), then we introduce a~representation ring of the quasi-local algebras. In Theorem \ref{thm:main2}, multiplicative KMS states are characterized in terms of this representation ring. In Section \ref{sect:3}, we apply the general theory in Section \ref{sect:2} to the quasi-local algebra given by quantum unitary groups. As a result, we give concrete examples of multiplicative quantized characters of ${\rm U}_q(\infty)$ (see Proposition \ref{prop:main3}). In Section \ref{sect:4}, we give a Gaussian fluctuation limit theorem with respect to multiplicative characters of ${\rm U}(\infty)$ (see Theorem \ref{theorem:gaussian}, Corollary \ref{cor:cha}).

\section{Multiplicative states on quasi-local algebras}\label{sect:2}
\subsection{Quasi-local algebras}
In this section, we discuss KMS states on a quasi-local algebra with flow. In Theorem \ref{thm:main1}, we give a necessary and sufficient condition that KMS states become \emph{multiplicative} in the sense of Definition \ref{def:multiplicative}. Throughout this section, we denote by $\mathcal{I}$ the set of all finite intervals of $\mathbb{N}$ or $\mathbb{Z}$.

Let $\mathfrak{M}$ be a unital $C^*$-algebra and $(M_I)_{I\in\mathcal{I}}$ a family of $W^*$-subalgebras in $\mathfrak{M}$, that is, $M_I$ is a $C^*$-subalgebra of $\mathfrak{M}$ and there exists a unique Banach space $(M_I)_*$ such that $(M_I)_*^*\cong M_I$. In the paper, we call $(\mathfrak{M}, (M_I)_{I\in\mathcal{I}})$ a \emph{quasi-local algebra} if
\begin{itemize}\itemsep=0pt
 \item[(ql1)] $M_I\subseteq M_J$ if $I\subseteq J$,
 \item[(ql2)] $\mathfrak{M}_0:=\bigcup_{I\in\mathcal{I}}M_I$ is dense in $\mathfrak{M}$,
 \item[(ql3)] $M_I$ has the common identity $1\in\mathfrak{M}$,
 \item[(ql4)] $M_I$ commutes with $M_J$ if $I\cap J=\varnothing$.
\end{itemize}
An element in $\mathfrak{M}_0$ is called a \emph{local observable} and $\mathfrak{M}_{0, {\rm sa}}$ denote the set of all self-adjoint local observables. We remark that the formalism here is slightly different from \cite[Section 2.6]{BR1}. In order to define quasi-local algebras, we use the relation on $\mathcal{I}$ given by the disjointness of intervals although it does not completely satisfy the condition of \emph{orthogonality relation} in \cite{BR1}. However, our formalism naturally appear in the asymptotic representation theory and is sufficiently suitable for the discussion in Section \ref{sect:4}.

Let $(\mathfrak{M}, (M_I)_{I\in\mathcal{I}})$ be a quasi-local algebra. We now follow the setup in \cite{Ueda22, Ueda20}. Namely, we assume that $\mathfrak{M}$ has a flow $\alpha\colon\mathbb{R}\curvearrowright\mathfrak{M}$ such that~$\alpha_t(M_I)=M_I$ and $\lim_{t\to0}\|\omega\circ\alpha_t|_{M_I}-\omega\|=0$ holds for any $I\in\mathcal{I}$ and $\omega\in (M_I)_*$. Let $\mathfrak{Z}_I$ denote the set of all minimal projections in the center $Z(M_I)$ of $M_I$ and assume that $\mathfrak{Z}_I$ is a countable set. Moreover, we assume that~$M_I=\ell^\infty\mathchar`-\bigoplus_{z\in\mathfrak{Z}_I}z M_I$ and $zM_I\cong B(\mathfrak{h}_z)$ for some \emph{finite dimensional} Hilbert space $\mathfrak{h}_z$. We particularly set $M_\varnothing=\mathbb{C}1$. By \cite[Lemma 7.1]{Ueda20}, the action $\alpha$ fixes elements in $Z(M_I)$, and hence for any $z\in\mathfrak{Z}_I$ the restriction $\alpha|_{zM_I}$ gives a continuous flow on $zM_I$. Now we fix an inverse temperature $\beta\in\mathbb{R}$, and then there is a unique normal $(\alpha|_{zM_I}, \beta)$-KMS state $\chi_z$ on $z M_I$.

\begin{Remark}\label{rem:2.1}
 We later deal with a quasi-local algebra given by the inductive system of the unitary groups ${\rm U}(I)$ on $\ell^2(I)\cong\mathbb{C}^{|I|}$, and then $M_I$ are given as the group $W^*$-algebras $W^*({\rm U}(I))$ (and $\alpha$ is trivial). This setting might seem slightly strange at a glance for the readers who are familiar with the asymptotic representation theory. In the asymptotic representation theory of~${\rm U}(\infty)$, we usually do not discriminate between ${\rm U}(I)$ and the unitary group $U(|I|)$ of rank~$|I|$. However, distinguishing them play an important role in the paper, and hence we deal with a~family of $W^*$-algebras indexed by $\mathcal{I}$.
\end{Remark}

For any $I\in\mathcal{I}$, we define a faithful normal conditional expectation $E_I\colon M_I\to Z(M_I)$ by
\[E_I(x):=\sum_{z\in\mathfrak{Z}_I}\chi_z(zx)z,\qquad x\in M_I.\]
\begin{Lemma}\label{lem:ocha}
 Let $I, J\in\mathcal{I}$ be disjoint and $I\sqcup J\in \mathcal{I}$. For any $x\in M_I$ and $y\in M_J$, we have~$E_{I\sqcup J}(xy)=E_{I\sqcup J}(E_I(x)E_J(y))$.
\end{Lemma}
\begin{proof}
 Let $x\in M_I$ and $y\in M_J$. Since every element in $M_I$ is a linear combination of (four) positive elements in $M_I$ (see, e.g., \cite[Proposition II.3.12\,(vi)]{Blackadar:book}), we may assume that $x$ is positive. For any $z\in\mathfrak{Z}_{I\sqcup J}$ with $zx\neq 0$, we define a state $\chi_{z, x}$ on $M_J$ by $\chi_{z, x}(y):=\chi_z(zxy)/\chi_z(zx)$ for any $y\in M_J$. Since $x\in (M_J)'$, it is easy to check $\chi_{z, x}$ is a normal $(\alpha|_{M_J}, \beta)$-KMS state on $M_J$. By \cite[Lemma 7.3]{Ueda20}, we have $\chi_{z, x}\circ E_J=\chi_{z, x}$, i.e., $\chi_z(zxE_J(y))=\chi_z(zxy)$ for any $z\in\mathfrak{Z}_{I\sqcup J}$. Thus, we have
 \[E_{I\sqcup J}(xE_J(y))=\sum_{z\in\mathfrak{Z}_{I\sqcup J}}\chi_z(zxE_J(y))z=\sum_{z\in\mathfrak{Z}_{I\sqcup J}}\chi_z(zxy)z=E_{I\sqcup J}(xy).\]
 Replacing the roles of $x$ and $y$, we have $E_{I\sqcup J}(xy)=E_{I\sqcup J}(E_I(x)y)$ for any $x\in M_I$, $y\in M_J$. Therefore, combining two formulas, $E_{I\sqcup J}(xy)=E_{I\sqcup J}(xE_J(y))=E_{I\sqcup J}(E_I(x)E_J(y))$ holds true for any $x\in M_I$, $y\in M_J$.
\end{proof}

We recall that a state $\chi$ on $\mathfrak{M}$ is said to be \emph{locally normal} if $\chi|_{{M_I}}$ are normal on $M_I$ for any~$I\in\mathcal{I}$. Moreover, we define the following:
\begin{Definition}\label{def:multiplicative}
 A state $\chi$ on $\mathfrak{M}$ is said to be \emph{multiplicative} if $I\cap J=\varnothing$ implies $\chi(xy)=\chi(x)\chi(y)$ for any $x\in M_I$ and $y\in M_J$.
\end{Definition}

Let us recall that $x\in\mathfrak{M}$ is said to be \emph{analytic} (for $\alpha$) if the function $t\in\mathbb{R}\mapsto \alpha_t(x)\in \mathfrak{M}$ can be extended to an entire function on $\mathbb{C}$. A state $\chi$ on $\mathfrak{M}$ is called a \emph{$(\alpha, \beta)$-KMS} (\emph{Kubo--Martin--Schwinger}) state if $\chi(x\alpha_{\mathrm{i}\beta}(y))=\chi(yx)$ holds for any analytic elements $x, y\in \mathfrak{M}$. We remark that $(\alpha, \beta)$-KMS states are nothing but tracial states when $\alpha$ is trivial.

The following is the first main result in the paper:
\begin{Theorem}\label{thm:main1}\samepage
 Let $\chi$ be a locally normal $(\alpha, \beta)$-KMS state on $\mathfrak{M}$. The following are equivalent:
 \begin{itemize}\itemsep=0pt
 \item[$(1)$] $\chi$ is multiplicative,
 \item[$(2)$] for any $I\sqcup J\in\mathcal{I}$ and $z\in Z(M_I)$, $w\in Z(M_J)$, we have $\chi(zw)=\chi(z)\chi(w)$.
 \end{itemize}
\end{Theorem}
\begin{proof}
 Clearly, (1) implies (2). We assume that the condition (2) holds. Let $I, J\in\mathcal{I}$ be disjoint and $x\in M_I$, $y\in M_J$. We may assume that $I\sqcup J\in\mathcal{I}$ by replacing $I$ or $J$ by bigger intervals if necessary. Then we have
 \begin{align*}
 \chi(xy)
 & =\chi(E_I(x)E_J(y))\qquad\qquad (\text{by \cite[Lemma 7.3]{Ueda20}, Lemma \ref{lem:ocha}}) \\
 & =\chi(E_I(x))\chi(E_J(y)) \\
 & =\chi(x)\chi(y)\hspace{25mm}(\text{by \cite[Lemma 7.3]{Ueda20}}).
 \end{align*}
 Namely, $\chi$ is multiplicative.
\end{proof}

\subsection{Quasi-local algebras with shift translations}\label{subsect:shift}
Here we assume the following additional setup: For every pair $I, J\in\mathcal{I}$, if $|I|=|J|$, then there is a normal $*$-isomorphism $\gamma_{I, J}\colon M_J\to M_I$ such that
\begin{itemize}\itemsep=0pt
 \item[(1)] $\gamma_{I, I}=\mathrm{id}$ and $\gamma_{I, J}\circ\gamma_{J, K}=\gamma_{I, K}$ for any $I, J, K\in\mathcal{I}$ with $|I|=|J|=|K|$,
 \item[(2)] if $I\subset J$, then $\gamma_{J+j, J}|_{M_I}=\gamma_{I+j, I}$ for any $j\in\mathbb{Z}$ with $J+j:=\{n\in\mathbb{Z}\mid n-j\in J\}\in\mathcal{I}$,
 \item[(3)] $\alpha_t|_{M_I}\circ\gamma_{I, J}=\gamma_{I, J}\circ\alpha_t|_{M_J}$ on $M_J$ for any $t\in\mathbb{R}$ and $I, J\in\mathcal{I}$ with $|I|=|J|$.
\end{itemize}
See Section \ref{sect:3} for a concrete example of this setting.

\begin{Lemma}
 For any $I, J\in\mathcal{I}$ with $|I|=|J|$, we have $E_I\circ\gamma_{I, J}=\gamma_{I, J}\circ E_J$ on $M_J$.
\end{Lemma}
\begin{proof}
 For any $z\in\mathfrak{Z}_J$ we have $\chi_{\gamma_{I, J}(z)}\circ\gamma_{I, J}|_{z M_J}=\chi_z$ by the uniqueness of a normal $(\alpha|_{zM_J}, \beta)$-KMS states on $zM_J$. Thus, for any $x\in M_J$ we have
 \[E_I(\gamma_{I, J}(x))=\sum_{w\in\mathfrak{Z}_I}\chi_w(w\gamma_{I, J}(x))w=\sum_{z\in\mathfrak{Z}_J}\chi_{\gamma_{I, J}(z)}(\gamma_{I, J}(zx))\gamma_{I, J}(z)=\gamma_{I, J}(E_J(x)).\tag*{\qed}
 \]\renewcommand{\qed}{}
\end{proof}

A state $\chi$ is said to be \emph{shift invariant} if $\chi\circ\gamma_{I, J}=\chi|_{M_J}$ on $M_J$ for any $I, J\in\mathcal{I}$ with~$|I|=|J|$. The aim here is to give a characterization of multiplicative shift invariant locally normal $(\alpha, \beta)$-KMS states.

Motivated by the work by Ueda \cite{Ueda22}, we introduce the following:
\begin{Definition}
 We define $\widetilde\Sigma$ as the algebraic direct sum $\bigoplus_{I\in\mathcal{I}}Z(M_I)$. The subspace $\Sigma$ is defined as the algebraic direct sum $\bigoplus_{n\geq0}Z(M_{[n]})$, where $[0]:=\varnothing$ and $[n]:=\{1,\dots, n\}$.
\end{Definition}
For any $z_I\in Z(M_I)$ and $z_J\in Z(M_J)$, we define $z_I\cdot z_J:=E_{I\vee J}(z_I\gamma_{J+j, J}(z_J))$, where $I\vee J:=I\sqcup(J+j)$ and $j\in\mathbb{Z}$ such that $\max I+1=\min J+j$. Namely, we have $I\vee J\in\mathcal{I}$. This operation defines an associative multiplication on $\widetilde\Sigma$. Let $z_*\in Z(M_*)$ for each $*=I, J, K\in\mathcal{I}$. For some $j, k\in\mathbb{Z}$, we have $(I\vee J)\vee K=I\vee (J\vee K)=I\sqcup (J+j)\sqcup (K+k)$. By Lemma \ref{lem:ocha}, we have
\begin{align*}
 (z_I\cdot z_J)\cdot z_K
 & =E_{(I\vee J)\vee K}(E_{I\vee J}(z_I\gamma_{J+j, J}(z_J))\gamma_{K+k, K}(z_K)) \\
 & =E_{I\vee(J\vee K)}(z_IE_{(J\vee K)+j}(\gamma_{J+j, J}(z_J)\gamma_{K+k, K}(z_K))) \\
 & =E_{I\vee(J\vee K)}(z_I\gamma_{(J\vee K)+j, J\vee K}(E_{J\sqcup(K+k-j)}(z_J\gamma_{K+k-j, K}(z_K)))) \\
 & =z_I\cdot(z_J\cdot z_K).
\end{align*}
Thus, $\widetilde\Sigma$ becomes a unital algebra over $\mathbb{C}$ with unit $1_\varnothing\in M_\varnothing=\mathbb{C}1_\varnothing$, and $\Sigma$ is a subalgebra of~$\widetilde{\Sigma}$.

\begin{Remark}
 As we mentioned in Remark \ref{rem:2.1}, unlike the usual discussion of asymptotic representation theory, we deal with $W^*$-algebras $M_I$ labeled by any finite intervals $I$ in $\mathbb{N}$ or $\mathbb{Z}$ in the paper. Hence it seems natural to introduce $\widetilde\Sigma$ rather than $\Sigma$. However, if a quasi-local algebra~$\mathfrak{M}$ has shift translations, then $M_I$ depends only on $|I|$, and one can expect $\Sigma$ to be sufficient to study appropriate (shift invariant) states. See Theorem \ref{thm:main2}. The unital algebra $\Sigma$ is defined similarly to \emph{representation rings} of Olshanski \cite{O16} and Ueda \cite{Ueda22}. If we deal with intervals of $\mathbb{N}$, there exists an ideal of $\Sigma$ such that the quotient should be understood as a dimension group of Vershik and Kerov (see \cite{KV80,VK85,VK87, KV90} and also \cite[Remark~3.4.1]{Ueda22}).
\end{Remark}

We define a linear map $\tilde\gamma\colon\widetilde\Sigma\to\widetilde\Sigma$ by $\tilde\gamma(z):=\gamma_{I+1, I}(z)$ for any $z\in Z(M_I)\subset\widetilde\Sigma$. Namely, we have $\tilde\gamma(1_\varnothing)=1_\varnothing$. Let $I, J\in\mathcal{I}$ such that $I\vee J=I\sqcup (J+j)$ and $z\in Z(M_I)$, $w\in Z(M_J)$. Since
\[(I+1)\vee J=(I\vee J)+1=(I+1)\sqcup(J+j+1),\qquad I\vee (J+1)=I\vee J,\]
we have
\begin{align*}z\cdot\tilde\gamma(w)&=z\gamma_{J+1, J}(w)=E_{I\vee (J+1)}(z\gamma_{J+j, J}(w))=z\cdot w,\\
 \tilde\gamma(z)\cdot w
 & =\gamma_{I+1, I}(z)\cdot w =E_{(I+1)\vee J}(\gamma_{I+1, I}(z)\gamma_{J+j+1, J}(w)) \\
 & =E_{(I\vee J)+1}(\gamma_{(I\vee J)+1, I\vee J}(z\gamma_{J+j, J}(w))) =\gamma_{(I\vee J)+1, I\vee J}(E_{I\vee J}(z\gamma_{J+j, J}(w))) \\
 & =\tilde\gamma(z\cdot w).
\end{align*}
Thus, we have $\tilde\gamma(z\cdot w)=\tilde\gamma(z)\cdot w=\tilde\gamma(z)\cdot\tilde\gamma(w)$, that is, $\tilde\gamma$ is a unital homomorphism. For~$z, w\in\widetilde\Sigma$ we write $z\sim w$ if $\tilde\gamma^k(z)=\tilde\gamma^l(w)$ for some $k, l\in\mathbb{N}$. Then $\sim$ is an equivalence relation on $\widetilde\Sigma$, and $[z]$ denotes the equivalence class of $z\in\widetilde\Sigma$. By the above two equations, $[z]\cdot[w]:=[z\cdot w]$, $\big(z, w\in\widetilde\Sigma\big)$ is a well-defined multiplication on $\widetilde\Sigma/{\sim}$, that is, $\widetilde\Sigma/{\sim}$ is also a unital algebra. Then we have following:
\begin{Proposition}
 Two algebras $\Sigma$ and $\widetilde\Sigma/{\sim}$ is isomorphic.
\end{Proposition}
\begin{proof}
 By definition, the quotient map $\widetilde\Sigma\to\widetilde\Sigma/{\sim}$ is multiplicative. Since $\Sigma\subset\widetilde\Sigma$, we naturally obtain a unital homomorphism $\Sigma\to\widetilde\Sigma/{\sim}$. Conversely, we obtain the inverse homomorphism~$\widetilde\Sigma/{\sim}\to \Sigma$ by $[z_I]\mapsto z$ for any $z_I\in Z(M_I)$, where $z$ is a unique element in $Z(M_{[|I|]})$ such that~$z_I\sim z$.
\end{proof}

For any linear functional $\chi$ on $\mathfrak{M}$, we define a linear functional $\tilde\omega_\chi$ on $\widetilde\Sigma$ by $\tilde\omega_\chi(z)=\chi(z)$ for any $z\in \widetilde\Sigma$. If $\chi$ is shift invariant, then we have $\tilde\omega_\chi\circ\tilde\gamma=\tilde\omega_\chi$. Thus, we obtain a linear functional $\omega_\chi$ on $\widetilde\Sigma/{\sim}$ by $\omega_\chi([z]):=\chi(z)$. If we identify $\widetilde\Sigma/{\sim}$ with $\Sigma$, then $\omega_\chi$ is nothing but the restriction of $\tilde\omega_\chi$ to $\Sigma$. Now we give a characterization of multiplicativity of locally normal shift invariant $(\alpha, \beta)$-KMS states:

\begin{Theorem}\label{thm:main2}
 Let $\chi$ be a shift invariant normal $(\alpha, \beta)$-KMS state on $(\mathfrak{M}, (M_I)_{I\in\mathcal{I}})$. The following are equivalent:
 \begin{itemize}\itemsep=0pt
 \item[$(1)$] $\chi$ is multiplicative,
 \item[$(2)$] $\tilde\omega_\chi$ is multiplicative on $\widetilde\Sigma$,
 \item[$(3)$] $\omega_\chi$ is multiplicative on $\Sigma$.
 \end{itemize}
\end{Theorem}
\begin{proof}
 The equivalence of claims (1) and (2) follows from Theorem \ref{thm:main1}. Moreover, (2) and (3) are equivalent since $\tilde\omega_\chi(z\cdot w)=\omega_\chi([z]\cdot [w])$ for any $z, w\in \widetilde\Sigma$.
\end{proof}

\section{The case of (quantum) unitary groups}\label{sect:3}
In this section, we give a quasi-local algebra by quantum unitary groups, which has naturally an $\mathbb{R}$-flow coming from the scaling automorphism groups of the quantum unitary groups. A~certain class of KMS-states with respect to this $\mathbb{R}$-flow relates to quantized characters of the infinite-dimensional quantum unitary group, which is a fundamental concept in the asymptotic representation theory for quantum groups (see \cite{Sato1,Sato3}). Using Theorem \ref{thm:main1}, we show that some quantized characters dealt in \cite{Sato4} give multiplicative states on the quasi-local algebra (see Proposition \ref{prop:main3}).

\subsection{Compact quantum groups}
We first recall basic facts of compact quantum groups. See \cite{KliSch,NeshveyevTuset} for more details.

Let $G$ be a \emph{compact quantum group}, that is, $G$ is a pair of unital $C^*$-algebra $C(G)$ and unital $*$-homomorphism $\delta_G\colon C(G)\to C(G)\otimes C(G)$ such that
\begin{itemize}\itemsep=0pt
 \item $(\delta_G\otimes \mathrm{id})\circ\delta_G=(\mathrm{id}\otimes \delta_G)\circ\delta_G$,
 \item $(C(G)\otimes 1)\delta_G(C(G)), (1\otimes C(G))\delta_G(C(G))$ are dense in $C(G)\otimes C(G)$,
\end{itemize}
where $\otimes$ denotes the operation of minimal tensor product of $C^*$-algebras. For a finite dimensional vector space $V$ an invertible element $U\in \mathrm{End}(V)\otimes C(G)$ is called a \emph{corepresentation} of $G$ if~$(\mathrm{id}\otimes \delta_G)(U)=U_{12}U_{13}$, where we use the leg numbering notation for $U_{ij}$. For any $\varphi\in\mathrm{End}(V)^*$ we call $(\varphi\otimes\mathrm{id})(U)\in C(G)$ the \emph{matrix coefficient} of $U$. It is known that the linear span of all matrix coefficients of finite dimensional corepresentations of $G$, denote by $\mathbb{C}[G]$, has a Hopf $*$-algebra structure. Unlike ordinary compact groups, the antipode $S_G$ of $\mathbb{C}[G]$ generally does not satisfy $S_G^2=\mathrm{id}$. In fact, for any finite dimensional corepresentation $U$ on $V$ there exists a positive invertible element $\rho_U\in\mathrm{End}(V)$ such that $\big(\mathrm{id}\otimes S_G^2\big)(U)=(\rho_U\otimes 1)U\big(\rho_U^{-1}\otimes 1\big)$. We assume that~$\rho_U$ also satisfies that $\mathrm{Tr}_V(\rho_U\cdot)=\mathrm{Tr}\big(\rho_U^{-1}\cdot\big)$ on the space of intertwining operators from~$U$ to itself. Under this assumption, $\rho_U$ is uniquely determined, and the \emph{scaling automorphism group}~$\tau^G\colon\mathbb{R}\curvearrowright\mathbb{C}[G]$ is defined by $\big(\mathrm{id}\otimes\tau^G_t\big)(U)=\big(\rho_U^{\mathrm{i}t}\otimes1\big)U\big(\rho_U^{-\mathrm{i}t}\otimes1\big)$.

Let $\mathcal{U}(G):=\mathbb{C}[G]^*$ be the dual space of $\mathbb{C}[G]$. Since $\mathbb{C}[G]$ is a Hopf $*$-algebra, $\mathcal{U}(G)$ naturally has a unital $*$-algebra structure. For any corepresentation $U$ on $V$, we obtain a $*$-representation~$\pi_U\colon\mathcal{U}(G)\to \mathrm{End}(V)$ by $\pi_U(x):=(\mathrm{id}\otimes x)(U)$. Let $\widehat G$ denote the set of all equivalence classes of irreducible corepresentations of $G$ and fix representatives ${\rm U}_\lambda$ for any $\lambda\in\widehat G$. It is known that the $*$-homomorphism $\pi_G\colon\mathcal{U}(G)\to \prod_{\lambda\in\widehat G}\mathrm{End}(V_\lambda)$ given by $\pi_G(x):=(\pi_{{\rm U}_\lambda}(x))$ is a $*$-isomorphism, where $V_\lambda$ is the representation space of ${\rm U}_\lambda$. Let $W^*(G)$ denote the unital $*$-subalgebra of $\mathcal{U}(G)$ isomorphic to
\[\ell^\infty\mathchar`-\bigoplus_{\lambda\in\widehat G}\mathrm{End}(V_\lambda):=\left\{(x_\lambda)_{\lambda\in\widehat G}\in\prod_{\lambda\in\widehat G}\mathrm{End}(V_\lambda)\Bigm| \sup_{\lambda\in\widehat G}\|x_\lambda\|<\infty\right\}\]
by $\pi_G$. By definition, $W^*(G)$ has a $W^*$-algebra structure and called the \emph{group $W^*$-algebra} of~$G$. We define the dual of the scaling automorphism group $\hat\tau^G\colon\mathbb{R}\curvearrowright \mathcal{U}(G)$ by $\hat\tau^G_t(x):=x\circ\tau^G_{-t}$. Clearly, we have $\pi_G\circ\hat\tau^G_t\circ\pi_G^{-1}=\prod_{\lambda\in\widehat G}\mathrm{Ad}\rho_{{\rm U}_\lambda}^{-\mathrm{i}t}$, and hence $\hat\tau^G_t(W^*(G))=W^*(G)$ for any $t\in\mathbb{R}$ and denote by the same symbol $\hat\tau^G$ the restriction of $\hat\tau^G$ to $W^*(G)$. In \cite{Sato1,Sato3}, we introduced the following notion, which is a fundamental concept in the asymptotic representation theory for quantum groups.

\begin{Definition}
 A normal $\big(\hat\tau^G, -1\big)$-KMS state on $W^*(G)$ is called a \emph{quantized character} of $G$.
\end{Definition}

\subsection{Quantum unitary groups}
Now we discuss the quantum unitary groups. We fix a quantization parameter $q\in (0, 1)$. Let~$\mathcal{I}$ be the set of all finite intervals in $\mathbb{N}$ or $\mathbb{Z}$ and $I\in\mathcal{I}$. We define a universal unital $*$-algebra~$\mathbb{C}[{\rm U}_q(I)]$ generated by $u_{ij}$ ($i, j\in I$) and $d_{q, I}^{-1}$ subject to the relations
\begin{gather*}
u_{ik}u_{jk}=qu_{jk}u_{ik},\qquad u_{ki}u_{kj}=qu_{kj}u_{ki},\qquad i<j,\\
u_{il}u_{jk}=u_{jk}u_{il},\qquad i<j,\quad k<l,\\
u_{ik}u_{jl}-u_{jl}u_{ik}=\big(q-q^{-1}\big)u_{jk}u_{il},\qquad i<j,\quad k<l,\\
d_{q, I}^{-1}u_{ij}=u_{ij}d_{q, I}^{-1}, \qquad d_{q, I}^{-1}d_{q, I}^I=d_{q, I}^Id_{q, I}^{-1}=1,\\
u_{ij}^*=(-q)^{j-i}d_{q, I\backslash\{j\}}^{I\backslash\{i\}}d_{q, I}^{-1}, \qquad d_{q, I}^{-1*}=d_{q, I}^I,
\end{gather*}
where for any $K=\{k_1<\dots<k_m\}$, $L=\{l_1<\dots <l_m\}\subseteq I$ we define
\[d_{q, L}^K:=\sum_{\sigma\in S_m}(-q)^{\ell(\sigma)}u_{k_{\sigma(1)}l_1}\cdots u_{k_{\sigma(m)}l_m},\]
$\ell(\sigma):=$ the number of inversions in $\sigma$.

The above defining relations are the same in the book \cite[Sections 9.2.1, 9.2.3, and 9.2.4]{KliSch}, but we use the different symbol $\mathbb{C}[{\rm U}_q(I)]$ instead of $\mathcal{O}({\rm U}_q(I))$ (or $\mathcal{O}(GL_q(I))$). It is known that~$\mathbb{C}[{\rm U}_q(I)]$ has a Hopf $*$-algebra structure and becomes a \emph{CQG algebra}. Thus, by \cite[Proposition~11.32]{KliSch}, there exists a compact quantum group, denote by ${\rm U}_q(I)$, such that its Hopf $*$-algebra of matrix coefficients of finite dimensional corepresentations is isomorphic to $\mathbb{C}[{\rm U}_q(I)]$, and its $C^*$-algebra~$C({\rm U}_q(I))$ is nothing but the universal $C^*$-algebra generated by $\mathbb{C}[{\rm U}_q(I)]$. We call ${\rm U}_q(I)$ the \emph{quantum} unitary group.

Let ${\rm U}(I)$ be the group of unitary operators on $\ell^2(I)\cong \mathbb{C}^{|I|}$. It is well known that ${\rm U}_q(I)$ and~${\rm U}(I)$ have the same representation theory. In particular, all irreducible (co-)representations of ${\rm U}_q(I)$ and ${\rm U}(I)$ are labeled by $\mathbb{S}_I:=\big\{\lambda=(\lambda_i)_{i\in I}\in\mathbb{Z}^I\mid \lambda\text{ is nonincreasing}\big\}$ (see \cite{KliSch, NoumiYamadaMimachi}). Thus, the group $W^*$-algebra $W^*({\rm U}_q(I))$ is $*$-isomorphic to $\ell^\infty\mathchar`-\bigoplus_{\lambda\in\mathbb{S}_I}\mathrm{End}(V_\lambda)$, where $V_\lambda$ is the representation space of the irreducible corepresentations of ${\rm U}_q(I)$ labeled by $\lambda$, and $V_\lambda$ has the same dimension of the irreducible representation of ${\rm U}(I)$ labeled by $\lambda$. Hence $W^*({\rm U}_q(I))$ is $*$-isomorphic to $W^*({\rm U}(I))$, but the dual scaling automorphism group $\hat\tau^{{\rm U}_q(I)}\colon\mathbb{R}\curvearrowright W^*({\rm U}_q(I))$ is nontrivial unlike in the case of $W^*({\rm U}(I))$. Moreover, we can describe $\hat\tau^{{\rm U}_q(I)}$ explicitly using the representation theory of the quantized universal enveloping algebra ${\rm U}_q\mathfrak{gl}_I$.

The \emph{quantized universal enveloping algebra} ${\rm U}_q\mathfrak{gl}_I$ is a universal unital algebra generated by $K_i$, $K_i^{-1}$ ($i\in I$) and $E_j$, $F_j$ ($j\in I\backslash\{\max I\}$) subject to the relations
\begin{gather*}
K_iK_i^{-1}=K_i^{-1}K_i=1,\qquad K_iK_j=K_jK_i,\\
K_iE_jK_i^{-1}=q^{\delta_{i, j}-\delta_{i, j+1}}E_j,\qquad K_iF_jK_i^{-1}=q^{\delta_{i, j}-\delta_{i, j+1}}F_j,\\
E_iF_j-F_jE_i=\delta_{i, j}\frac{K_iK_{i+1}^{-1}-K_i^{-1}K_{i+1}}{q-q^{-1}},\\
E_iE_j=E_jE_i, \qquad F_iF_j=F_jF_i,\qquad |i-j|\geq2,\\
E_i^2E_{i\pm 1}-\big(q+q^{-1}\big)E_iE_{i\pm 1}E_i+E_{i\pm 1}E_i^2=0,\\
F_i^2F_{i\pm 1}-\big(q+q^{-1}\big)F_iF_{i\pm 1}F_i+F_{i\pm 1}F_i^2=0.
\end{gather*}
Here we follow the notation in \cite{KliSch}. It is known that ${\rm U}_q\mathfrak{gl}_I$ has a Hopf $*$-algebra structure, and there exists a non-degenerate dual pairing between two Hopf $*$-algebras ${\rm U}_q\mathfrak{gl}_I$ and $\mathbb{C}[{\rm U}_q(I)]$ (see \cite[Theorem 9.18, Corollary 11.54]{KliSch}). Using this dual pairing, we can identify ${\rm U}_q\mathfrak{gl}_I$ with a $*$-subalgebra of $\mathcal{U}({\rm U}_q(I))$, and hence for any corepresentation ${\rm U}$ of ${\rm U}_q(I)$ on $V$ we obtain a~$*$-representation $\pi_U\colon {\rm U}_q\mathfrak{gl}_I\to\mathrm{End}(V)$, which is the restriction of $\pi_U\colon\mathcal{U}({\rm U}_q(I))\to\mathrm{End}(V)$ to~${\rm U}_q\mathfrak{gl}_I$. Let ${\rm U}_\lambda$ be an irreducible corepresentation of ${\rm U}_q(I)$ labeled by $\lambda\in\mathbb{S}_I$. Then $(\pi_{{\rm U}_\lambda}, V_\lambda)$ is a highest weight representation with highest weight $\lambda$ (see \cite[Proposition 11.50]{KliSch}), and~$\pi_{{\rm U}_\lambda}$ is hereinafter referred to as $\pi_\lambda$. Moreover, we have $\rho_{{\rm U}_\lambda}=\pi_\lambda\big(K_{i_1}^{-|I|+1}K_{i_2}^{-|I|+3}\cdots K_{i_{|I|}}^{|I|-1}\big)$, where~${I=\{i_1<\cdots<i_{|I|}\}}$ (i.e., $i_{|I|}=i_1+|I|-1$). We also have
\[d_q(\lambda):=\mathrm{Tr}_{V_\lambda}(\rho_{{\rm U}_\lambda})=\mathrm{Tr}_{V_\lambda}\big(\rho_{{\rm U}_\lambda}^{-1}\big)=s_\lambda\big(q^{|I|-1}, q^{|I|-3}, \dots, q^{-|I|+1}\big),\]
where $s_\lambda$ is the Schur (Laurent) polynomial labeled by $\lambda$.

\subsection{The quasi-local algebra from quantum unitary groups}
Here we construct a quasi-local algebra from the quantum unitary groups ${\rm U}_q(I)$. First, we check that the group $W^*$-algebras $W^*({\rm U}_q(I))$ form an inductive system, and then we take their $C^*$-inductive limit $\mathfrak{M}({\rm U}_q)$. In Proposition \ref{prop:qlalg}, we show that the pair $(\mathfrak{M}({\rm U}_q), (W^*({\rm U}_q(I))_{I\in\mathcal{I}}))$ is a~quasi-local algebra.

Let $I, J\in\mathcal{I}$ such that $I\subseteq J$. By definition, there is a surjective unital $*$-homomorphism $\theta^J_I\colon \mathbb{C}[{\rm U}_q(J)]\to\mathbb{C}[{\rm U}_q(I)]$ such that for any $i, j\in J$
\[\theta^J_I(u_{ij}):=\begin{cases}u_{ij}, & i, j\in I,\\
\delta_{i, j}1,&\text{otherwise},
\end{cases}
\qquad \theta^J_I\big(d_{q, J}^{-1}\big):=d_{q, I}^{-1},\]
and $\delta_{{\rm U}_q(I)}\circ\theta^J_I=\big(\theta^J_I\otimes \theta^J_I\big)\circ\delta_{{\rm U}_q(J)}$ holds true. Namely, ${\rm U}_q(I)$ is an \emph{algebraic} quantum subgroup of~${\rm U}_q(J)$. By \cite[Theorem 2.7.10]{NeshveyevTuset}, ${\rm U}_q(I)$ is co-amenable, and hence ${\rm U}_q(I)$ is a quantum subgroup of ${\rm U}_q(J)$ (see \cite[Lemma 2.7]{Tomatsu}). Namely, $\theta^J_I$ extends to a surjective unital $*$-homomorphism from~$C({\rm U}_q(J))$ to $C({\rm U}_q(I))$. The dual map $\Theta^J_I\colon\mathcal{U}({\rm U}_q(I))\to \mathcal{U}({\rm U}_q(J))$ given by $\Theta^J_I(x):=x\circ\theta^J_I$ is an injective unital $*$-homomorphism, and we have
\[\Theta^J_I(W^*({\rm U}_q(I)))\subset W^*({\rm U}_q(J)),\qquad \Theta^J_I\circ\hat\tau^{{\rm U}_q(I)}_t=\hat\tau^{{\rm U}_q(J)}_t\circ\Theta^J_I.\]
See, e.g., \cite{Sato1} for more details. If $I\subseteq J\subseteq K$, then $\Theta^K_J\circ\Theta^J_I=\Theta^K_I$ holds true. Thus, the group $W^*$-algebras $W^*({\rm U}_q(I))$ form an inductive system, and $\mathfrak{M}({\rm U}_q)$ denotes their $C^*$-inductive limit. We remark that the canonical $*$-homomorphism $W^*({\rm U}_q(I))\to\mathfrak{M}({\rm U}_q)$ is injective since~$\Theta^J_I$ is injective for any $I, J\in\mathcal{I}$ with $I\subseteq J$. Thus, we can freely identify $W^*({\rm U}_q(I))$ with a~subalgebra of $\mathfrak{M}({\rm U}_q)$. By the universality of~$\mathfrak{M}({\rm U}_q)$, we obtain a unique flow $\hat\tau\colon\mathbb{R}\curvearrowright\mathfrak{M}({\rm U}_q)$ such that~\smash{$\hat\tau_t|_{W^*({\rm U}_q(I))}=\hat\tau^{{\rm U}_q(I)}_t$} for any $t\in\mathbb{R}$ and $I\in\mathcal{I}$.

For any $I, J\in \mathcal{I}$, the tensor product $\mathbb{C}[{\rm U}_q(I)]\otimes\mathbb{C}[{\rm U}_q(J)]$ is also a CQG algebra, and hence we obtain the associated compact quantum group, denoted by ${\rm U}_q(I)\times {\rm U}_q(J)$. We have
\[\mathcal{U}({\rm U}_q(I)\times {\rm U}_q(J)):=(\mathbb{C}[{\rm U}_q(I)]\otimes\mathbb{C}[{\rm U}_q(J)])^*\cong \prod_{\lambda\in\mathbb{S}_I, \mu\in\mathbb{S}_J}\mathrm{End}(V_\lambda)\otimes\mathrm{End}(V_\mu),\]
and hence $W^*({\rm U}_q(I)\times {\rm U}_q(J))\cong W^*({\rm U}_q(I))\bar\otimes W^*({\rm U}_q(J))$ taking the bounded parts of the above two $*$-algebras, where $\bar\otimes$ denotes the operation of tensor products of $W^*$-algebras (see, e.g., \cite[Section III.1.5]{Blackadar:book}).

\begin{Lemma}\label{lem:commute}
 Let $I, J\in\mathcal{I}$. If $I\cap J=\varnothing$, then $W^*({\rm U}_q(I))$ commutes with $W^*({\rm U}_q(J))$. Namely, $(\mathfrak{M}({\rm U}_q), ({\rm U}_q(I))_{I\in\mathcal{I}})$ satisfies condition $(ql4)$.
\end{Lemma}
\begin{proof}
 Let $K\in\mathcal{I}$ such that $I\sqcup J\subseteq K$. By $\big(\theta^K_I\otimes \theta^K_J\big)\circ\delta_{{\rm U}_q(K)}$, we can regard ${{\rm U}_q(I)\times {\rm U}_q(J)}$ as a~quantum subgroup ${\rm U}_q(K)$. Therefore, we obtain an injective $*$-homomorphism $\Theta^K_{I, J}$ from $W^*({\rm U}_q(I))\bar\otimes W^*({\rm U}_q(J))$ to $W^*({\rm U}_q(K))$ such that for any $x\in W^*({\rm U}_q(I))$ and $y\in W^*({\rm U}_q(J))$
 \[\big(\Theta^K_{I, J}\big)(x\otimes y)=(x\otimes y)\circ \big(\theta^K_I\otimes \theta^K_J\big)\circ\delta_{{\rm U}_q(K)}=\Theta^K_I(x)\Theta^K_J(y).\]
 Namely, $\Theta^K_I(W^*({\rm U}_q(I)))$ and $\Theta^K_J(W^*({\rm U}_q(J)))$ mutually commute in $W^*({\rm U}_q(K))$.
\end{proof}

\begin{Proposition}\label{prop:qlalg}
 The pair $(\mathfrak{M}({\rm U}_q), (W^*({\rm U}_q(I)))_{I\in\mathcal{I}})$ is a quasi-local algebra.
\end{Proposition}
\begin{proof}
 The first two conditions (ql1), (ql2) are clear. Since $\Theta^J_I$ is unital for any $I, J\in\mathcal{I}$ with $I\subseteq J$, any $W^*({\rm U}_q(I))$ have a common identity $1\in\mathfrak{M}$, that is, (ql3) holds. The condition (ql4) follows from Lemma \ref{lem:commute}.
\end{proof}

The quasi-local algebra $(\mathfrak{M}({\rm U}_q), (W^*({\rm U}_q(I)))_{I\in\mathcal{I}})$ also admits shift transformations.
\begin{Lemma}
 For every pair $I, J\in\mathcal{I}$ such that $|I|=|J|$, there exists a normal $*$-isomorphism $\gamma_{I, J}\colon W^*({\rm U}_q(J))\to W^*({\rm U}_q(I))$ satisfying $(\gamma 1)$, $(\gamma 2)$, $(\gamma 3)$ with respect to $\hat\tau$.
\end{Lemma}
\begin{proof}
 We remark that there is $k\in\mathbb{Z}$ such that $I=J+k$. By definition, we obtain a unital $*$-isomorphism $r_{I, J}\colon \mathbb{C}[{\rm U}_q(I)]\to\mathbb{C}[{\rm U}_q(J)]$ by $r_{I, J}(u_{i+k, j+k}):=u_{ij}$ ($i, j\in J$) and $r_{I, J}\big(d_{q, I}^{-1}\big)=d_{q, J}^{-1}$, and we have
 \[\delta_{{\rm U}_q(J)}\circ r_{I, J}=(r_{I, J}\otimes r_{I, J})\circ\delta_{{\rm U}_q(I)}, \qquad \delta_{{\rm U}_q(I)}\circ r_{J, I}=(r_{J, I}\otimes r_{J, I})\circ\delta_{{\rm U}_q(J)},\]
 where $r_{J, I}:=r_{I, J}^{-1}$. Thus, ${\rm U}_q(I)$ and ${\rm U}_q(J)$ can be regarded as a quantum subgroup of each other. Therefore, there exists an injective normal $*$-isomorphism $\gamma_{I, J}\colon W^*({\rm U}_q(J))\to W^*({\rm U}_q(I))$ satisfying $(\gamma 3)$, that is, \smash{$\hat\tau^{{\rm U}_q(I)_t}\circ\gamma_{I, J}=\gamma_{I, J}\circ\hat\tau^{{\rm U}_q(J)}_t$} for any $t\in\mathbb{R}$. Let $I, J, K\in\mathcal{I}$. If ${|I|=|J|=|K|}$, then we have $\gamma_{I, J}\circ \gamma_{J, K}=\gamma_{I, K}$ since $r_{J, K}\circ r_{I, J}=r_{I, K}$. Namely, $(\gamma 1)$ holds true. If $I\subseteq J$, then we have $\theta^{J+k}_{I+k}\circ r_{J, J+k}=r_{I, I+k}\circ\theta^J_I$ for any $k\in\mathbb{Z}$ such that $J+k\in\mathcal{I}$. Thus, we have \smash{$\gamma_{J, J+k}\circ \Theta_{I+k}^{J+k}=\Theta^J_I\circ \gamma_{I, I+k}$}, that is, $(\gamma 2)$ holds true.
\end{proof}

Here we recall the observation by Ueda \cite[Section 4.4.5]{Ueda22}. We remark that the direction of the parameter $t$ in \smash{$\big(\hat\tau^{{\rm U}_q(I)}_t\big)_{t\in\mathbb{R}}$} is inverse from \cite{Ueda22}. Due to this, the parameter $q$ is replaced by $q^{-1}$ in the following computation. Since $W^*({\rm U}_q(I))\cong\ell^\infty\mathchar`-\bigoplus_{\lambda\in\mathbb{S}_I}\mathrm{End}(V_\lambda)$, the center~$Z(W^*({\rm U}_q(I)))$ of $W^*({\rm U}_q(I))$ is $*$-isomorphic to $\ell^\infty(\mathbb{S}_I)$. Let $z_\lambda$ ($\lambda\in\mathbb{S}_I$) denote the minimal projection in $Z(W^*({\rm U}_q(I)))$ corresponding to the Dirac function $\delta_\lambda\in\ell^\infty(\mathbb{S}_I)$. Namely, we have~$\pi_\mu(z_\lambda)=\delta_{\lambda, \mu}1_{\mathrm{End}(V_\mu)}$ for any $\mu\in\mathbb{S}_I$ and $z_\lambda W^*({\rm U}_q(I))\cong \mathrm{End}(V_\lambda)$. Moreover, the unique~\smash{$\big(\hat\tau^{{\rm U}_q(I)}|_{z_\lambda W^*({\rm U}_q(I))}, -1\big)$}-KMS state $\chi_{z_\lambda}$ is given by
\[\chi_{z_\lambda}(x):=\frac{\mathrm{Tr}_{V_\lambda}\big(\rho_{{\rm U}_\lambda}^{-1}\pi_\lambda(x)\big)}{\mathrm{Tr}_{V_\lambda}\big(\rho_{{\rm U}_\lambda}^{-1}\big)}=\frac{\mathrm{Tr}_{V_\lambda}\big(\rho_{{\rm U}_\lambda}^{-1}\pi_\lambda(x)\big)}{d_q(\lambda)},\qquad x\in z_\lambda W^*({\rm U}_q(I)).\]
Now we assume that $I, J\in\mathcal{I}$ are disjoint and $I\sqcup J$ also in $\mathcal{I}$. Let $I\sqcup J=\{i_1<\cdots <i_{|I|+|J|}\}$ and $I=\{i_1<\cdots <i_{|I|}\}$. Namely, we have $J=\{i_{|I|+1}<\cdots <i_{|I|+|J|}\}$. For any $\nu\in\mathbb{S}_{I\sqcup J}$ the restriction of $(\pi_\nu, V_\nu)$ to ${\rm U}_q\mathfrak{gl}_I\otimes {\rm U}_q\mathfrak{gl}_J$ decompose into the direct sum
\[\bigoplus_{\lambda\in\mathbb{S}_I, \mu\in\mathbb{S}_J}(\pi_\lambda\otimes\pi_\mu, V_\lambda\otimes V_\mu)^{\oplus c^\nu_{\lambda, \mu}},\]
where the multiplicity $c^\nu_{\lambda, \mu}$ is the same in the classical case and determined by
\begin{equation}\label{eq:panpkin}
 s_\nu(x_1,\dots,x_{|I|+|J|})=\sum_{\lambda\in\mathbb{S}_I, \mu\in\mathbb{S}_J}c^\nu_{\lambda, \mu}s_\lambda(x_1,\dots, x_{|I|})s_\mu(x_{|I|+1},\dots, x_{|I|+|J|}).
\end{equation}
See, e.g., \cite[Proposition 5.4]{DS}. Therefore, for any $\lambda\in\mathbb{S}_I$ and $\mu\in\mathbb{S}_J$ we have
\begin{align}\label{eq:jenny}
 \chi_{z_\nu}(z_\lambda z_\mu)
 & =\frac{1}{d_q(\nu)}\mathrm{Tr}_{V_\lambda}\big(\pi_\nu\big(K_{i_1}^{|I|+|J|-1}\cdots K_{i_{|I|+|J|}}^{-|I|-|J|+1}\big)\pi_\nu(z_\lambda z_\mu)\big)\nonumber \\
 & =\frac{c^\nu_{\lambda, \mu}}{d_q(\nu)}\mathrm{Tr}_{V_\lambda}\big(\pi_\lambda\big(K_{i_1}^{|I|+|J|-1}\cdots K_{i_{|I|}}^{-|I|+|J|+1}\big)\big)\nonumber\\
 &\quad \times \mathrm{Tr}_{V_\mu}\big(\pi_\mu\big(K_{i_{|I|+1}}^{-|I|+|J|-1}\cdots K_{i_{|I|+|J|}}^{-|I|-|J|+1}\big)\big)\nonumber \\
 & =\frac{c^\nu_{\lambda, \mu}}{d_q(\nu)}s_\lambda\big(q^{|I|+|J|-1},\dots,q^{-|I|+|J|+1}\big)s_\mu\big(q^{-|I|+|J|-1}, \dots, q^{-|I|-|J|+1}\big)\nonumber \\
 & =c^\nu_{\lambda, \mu}q^{|\lambda||J|-|\mu||I|}\frac{d_q(\lambda)d_q(\mu)}{d_q(\nu)},
\end{align}
where $|\lambda|:=\lambda_{i_1}+\dots+\lambda_{i_{|I|}}$ for any $\lambda\in \mathbb{S}_I$ and $I=\{i_1<\dots <i_{|I|}\}$.

\subsection[Examples of multiplicative states on M(U\_q)]{Examples of multiplicative states on $\boldsymbol{\mathfrak{M}({\rm U}_q)}$}\label{sect:mult}
Here we give examples of multiplicative states on the quasi-local algebra from the quantum unitary groups ${\rm U}_q(I)$. In this section, let $\mathcal{I}$ be the set of all finite intervals of $\mathbb{N}$,

We first recall some facts about characters of ${\rm U}(\infty)=\varinjlim_{I\in\mathcal{I}}{\rm U}(I)$. We endow ${\rm U}(\infty)$ with the inductive limit topology, that is, a function $f$ on ${\rm U}(\infty)$ is continuous if and only if $f|_{{\rm U}(I)}$ is continuous on ${\rm U}(I)$ for any $I\in\mathcal{I}$. A continuous function $f$ on ${\rm U}(\infty)$ is called a \emph{character} of~${\rm U}(\infty)$ if $f$ is positive-definite, central (i.e., $f(uv)=f(vu)$ for any $u, v\in {\rm U}(\infty)$), and $f(e)=1$, where $e\in {\rm U}(\infty)$ is the identity element. By definition, the set of all characters of~${\rm U}(\infty)$ is a convex set, and it is known that every extreme character $f$ of ${\rm U}(\infty)$ is multiplicative, that is, $f(uv)=f(u)f(v)$ for any $u\in {\rm U}(I)$ and $v\in U(J)$ if $I, J\in\mathcal{I}$ are disjoint (see \cite{Voiculescu74}). Moreover, it is known the following explicit description of the extreme characters of ${\rm U}(\infty)$ (see~\cite{BO12,Boyer83,OkounkovOlshanski,Petrov14,VK82,Voiculescu76}): There exists a bijection between the set of all extreme characters of~${\rm U}(\infty)$ and the set $\Omega$ consisting of~$\omega=\big(\alpha^+, \beta^+, \alpha^-, \beta^-, \gamma^+, \gamma^-\big)\in(\mathbb{R}_{\geq0}^\infty)^4\times\mathbb{R}_{\geq0}^2$ such that
\begin{gather*}
\alpha^\pm=\big(\alpha^\pm_1\geq\alpha^\pm_2\geq\cdots\big), \qquad \beta^\pm=\big(\beta^\pm_1\geq\beta^\pm_2\geq\cdots\big),\\
\sum_{i=1}^\infty\big(\alpha^\pm_i+\beta^\pm_i\big)<\infty, \qquad \beta^+_1+\beta^-_1\leq1.\end{gather*}
More precisely, for any $\omega\in \Omega$ the corresponding extreme character $f_\omega$ is given as
\[f_\omega(u):=\prod_{z}\Phi_\omega(z),\qquad u\in {\rm U}(\infty),\]
where $z$ runs over all eigenvalues of $u$ and
\[\Phi_\omega(z):={\rm e}^{\gamma^+(z-1)+\gamma^-(z^{-1}-1)}\prod_{i=1}^\infty\frac{1+\beta^+_i(z-1)}{1-\alpha^+_i(z-1)}\frac{1+\beta^-_i(z^{-1}-1)}{1-\alpha^-_i(z^{-1}-1)}.\]

For $q\in (0, 1]$, let $\Omega_q$ denote the set of $\omega\in \Omega$ satisfying that $\Phi_\omega\big(q^{-2i}\big)>0$ for any $i\in\mathbb{Z}_{\geq0}$ and the Laurent expansion of $\Phi_\omega(z)$ converges in a domain containing $q^{-2i}$ for all $i\in\mathbb{Z}_{\geq0}$. Clearly, we have $\Omega_1=\Omega$. By \cite[Lemma 7.1]{Sato4}, there exists a unique locally normal $(\hat\tau, -1)$-KMS state~$\chi_\omega^{(q)}$ on $\mathfrak{M}({\rm U}_q)$ such that for any $I\in\mathcal{I}$ and $(z_i)_{i\in I}\in\mathbb{T}^I$
\begin{equation}\label{eq:char_omega}
 \prod_{i\in I}\frac{\Phi_\omega\big(q^{-2(i-1)}z_i\big)}{\Phi_\omega\big(q^{-2(i-1)}\big)}=
 \sum_{\lambda\in\mathbb{S}_I}\frac{\chi_\omega^{(q)}(z_\lambda)}
 {d_q(\lambda)}s_\lambda\big(q^{|I|-1}z_{i_1}, q^{|I|-3}z_{i_2},\dots, q^{-|I|+1}z_{i_{|I|}}\big),
\end{equation}
where $d_q(\lambda):=s_\lambda\big(q^{|I|-1}, q^{|I|-3},\dots, q^{-|I|+1}\big)$ and $I=\{i_1<i_2<\dots<i_{|I|}\}$.

\begin{Remark}
 For $q=1$, let $W^*({\rm U}_q(I))=W^*({\rm U}(I))$. In this case, $\hat\tau$ is trivial, and hence $(\hat\tau, -1)$-KMS states are nothing but tracial states on $\mathfrak{M}(U):=\varinjlim_{I\in\mathcal{I}}W^*({\rm U}(I))$. We recall that there exists a bijection between the set of all locally normal tracial states on $\mathfrak{M}(U)$ and the set of all characters of ${\rm U}(\infty)$. Moreover, $\chi_\omega^{(1)}$ corresponds to the characters $f_\omega$.
\end{Remark}

\begin{Remark}
 The $W^*$-inductive limit of $(W^*({\rm U}_q(I)))_{I\in\mathcal{I}}$ naturally has a structure like a~Woro\-nowicz algebra, and hence we can regard this as an inductive limit quantum group, called the \emph{infinite-dimensional quantum unitary group} (see \cite{Sato3}). The $\mathbb{R}$-flow $\hat\tau$ extends to the $W^*$-inductive limit, and a normal $(\hat\tau, -1)$-KMS state is called a \emph{quantized character}. Moreover, there exists a~bijection between the quantized characters and the locally normal $(\hat\tau, -1)$-KMS states on~$\mathfrak{M}({\rm U}_q)$.
\end{Remark}

\begin{Proposition}\label{prop:main3}
 For any $q\in(0, 1]$ and $\omega\in\Omega_q$, the corresponding state $\chi_\omega^{(q)}$ on $\mathfrak{M}({\rm U}_q)$ is multiplicative.
\end{Proposition}
\begin{proof}
 By Theorem \ref{thm:main1}, it suffices to show that $\chi_\omega^{(q)}(z_\lambda z_\mu)=\chi_\omega^{(q)}(z_\lambda)\chi_\omega^{(q)}(z_\mu)$ for any $\lambda\in\mathbb{S}_I$, $\mu\in\mathbb{S}_J$ if $I\sqcup J\in\mathcal{I}$. We suppose that $I\sqcup J=\{i_1<\dots <i_{|I|+|J|}\}$ and $I=\{i_1<\dots <i_{|I|}\}$. Namely, we have $J=\{i_{|I|+1}<\dots <i_{|I|+|J|}\}$. By equations \eqref{eq:panpkin}, \eqref{eq:jenny}, \eqref{eq:char_omega}, we have
 \begin{align*}
 & \prod_{i\in I\sqcup J}\frac{\Phi_\omega\big(q^{-2(i-1)}z_i\big)}{\Phi_\omega\big(q^{-2(i-1)}\big)} \\ &\qquad=\sum_{\nu\in\mathbb{S}_{I\sqcup J}}\frac{\chi_\omega^{(q)}(z_\nu)}{d_q(\nu)}s_\nu\big(q^{|I|+|J|-1}z_{i_1},\dots, q^{-|I|-|J|+1}z_{i_{|I|+|J|}}\big) \\
 &\qquad =\sum_{\lambda\in\mathbb{S}_I, \mu\in\mathbb{S}_J}\bigg(\sum_{\nu\in\mathbb{S}_{I\sqcup J}}\chi_\omega^{(q)}(z_\nu)\frac{c^\nu_{\lambda, \mu}}{d_q(\nu)}\bigg)q^{|\lambda||J|-|\mu||I|} \\
 & \qquad\qquad\qquad \times s_\lambda\big(q^{|I|-1}z_{i_1},\dots, q^{-|I|+1}z_{i_{|I|}}\big)s_\mu\big(q^{|J|-1}z_{i_{|I|+1}}, \dots, q^{-|J|+1}z_{i_{|I|+|J|}}\big) \\
 &\qquad =\sum_{\lambda\in\mathbb{S}_I, \mu\in\mathbb{S}_J}\frac{1}{d_q(\lambda)d_q(\mu)}\bigg(\sum_{\nu\in\mathbb{S}_{I\sqcup J}}\chi_\omega^{(q)}(z_\nu)\chi_{z_\nu}(z_\lambda z_\mu)\bigg) \\
 &\qquad \qquad\qquad \times s_\lambda\big(q^{|I|-1}z_{i_1},\dots, q^{-|I|+1}z_{i_{|I|}}\big)s_\mu\big(q^{|J|-1}z_{i_{|I|+1}}, \dots, q^{-|J|+1}z_{i_{|I|+|J|}}\big) \\
 & \qquad=\sum_{\lambda\in\mathbb{S}_I, \mu\in\mathbb{S}_J}\frac{\chi_\omega^{(q)}(z_\lambda z_\mu)}{d_q(\lambda)d_q(\mu)} \\
 &\qquad \qquad\qquad \times s_\lambda\big(q^{|I|-1}z_{i_1},\dots, q^{-|I|+1}z_{i_{|I|}}\big)s_\mu\big(q^{|J|-1}z_{i_{|I|+1}}, \dots, q^{-|J|+1}z_{i_{|I|+|J|}}\big).
 \end{align*}
 On the other hand, we have
 \begin{align*}
 & \prod_{i\in I\sqcup J}\frac{\Phi_\omega\big(q^{-2(i-1)}z_i\big)}{\Phi_\omega\big(q^{-2(i-1)}\big)} \\
 &\qquad =\sum_{\lambda\in\mathbb{S}_I, \mu\in\mathbb{S}_J}\frac{\chi_\omega^{(q)}(z_\lambda)\chi_\omega^{(q)}(z_\mu)}{d_q(\lambda)d_q(\mu)}\\
 &\qquad\qquad\times s_\lambda\big(q^{|I|-1}z_{i_1},\dots, q^{-|I|+1}z_{i_{|I|}}\big)s_\mu\big(q^{|J|-1}z_{i_{|I|+1}}, \dots, q^{-|J|+1}z_{i_{|I|+|J|}}\big).
 \end{align*}
 Therefore, we have $\chi_\omega^{(q)}(z_\lambda z_\mu)=\chi_\omega^{(q)}(z_\lambda)\chi_\omega^{(q)}(z_\mu)$.
\end{proof}

The restriction of $\chi_\omega^{(q)}$ to $W^*({\rm U}_q(I))$ is a quantized character of ${\rm U}_q(I)$. By equation \eqref{eq:char_omega}, it is multiplicative on the ``maximal tori'' of ${\rm U}_q(I)$. By Proposition \ref{prop:main3}, the state $\chi_\omega^{(q)}$, which gives a~quantized characters having a multiplicative form on the maximal tori, is really multiplicative in the sense of Definition \ref{def:multiplicative}. In the classical case ($q=1$), the multiplicative characters are extreme characters. On the other hand, most of $\chi_\omega^{(q)}$ is not extreme quantized character if~${0<q<1}$. Thus, the $\chi_\omega^{(q)}$ and Proposition \ref{prop:main3} seem to indicate an essential difference in the character theory of ${\rm U}(\infty)$ and ${\rm U}_q(\infty)$.

\section{Quasi-local algebras and Gaussian fluctuation limits}\label{sect:4}
In this section, we mention a central limit theorem with respect to multiplicative states on quasi-local algebras (see Theorem \ref{theorem:gaussian}), and its proof is given in Appendix \ref{app:proof}. We mainly refer to the book by Petz \cite{Petz:book} about Weyl CCR algebras and quasi-free states.

Throughout this section and Appendix \ref{app:proof}, let $\mathcal{I}$ be the set of all finite intervals in $\mathbb{Z}$ and $(\mathfrak{M}, (M_I)_{I\in\mathcal{I}})$ a quasi-local algebra. Moreover, we follow the same setup in Section \ref{sect:2}. Namely, $\mathfrak{M}$ has an action $\alpha\colon \mathbb{R}\curvearrowright\mathfrak{M}$ such that $\alpha_t(M_I)=M_I$ and $\lim_{t\to0}\|\omega\circ\alpha_t-\omega\|=0$ hold for any~$I\in\mathcal{I}$ and $\omega\in(M_I)_*$. In addition, we suppose that $\mathfrak{M}$ admits a shift translations. We remark that this assumption is equivalent to the existence of action $\gamma\colon \mathbb{Z}\curvearrowright \mathfrak{M}$ such that $\gamma$ commutes with~$\alpha$ and~$\gamma_j(M_I)=M_{I+j}$ holds for any $I\in\mathcal{I}$ and $j\in\mathbb{Z}$. Clearly, $\gamma_j$ preserves~$\mathfrak{M}_0:=\bigcup_{I\in\mathcal{I}}M_I$ for every $j\in\mathbb{Z}$. Throughout this section, we fix a $\gamma$-invariant multiplicative state $\chi$ on $\mathfrak{M}$. Moreover, we fix an increasing sequence $(I_n)_{n=1}^\infty$ in $\mathcal{I}$ such that $\bigcup_{n=1}^\infty I_n=\mathbb{Z}$.

For any $x\in\mathfrak{M}$ and $I\in\mathcal{I}$ the \emph{local fluctuation} $F_I(x)$ of $x$ on $I$ is defined by
\[F_I(x):=\frac{1}{\sqrt{|I|}}\sum_{j\in I}(\gamma_j(x)-\chi(x)1).\]
Let $F_n(x):=F_{I_n}(x)$ for every $n=1, 2, \dots$. Since $\chi(x\gamma_j(y))=\chi(x)\chi(y)$ if $|j|$ is sufficiently large, we have $\sum_{j\in\mathbb{Z}}|\chi(x\gamma_j(y))-\chi(x)\chi(y)|<\infty$. Thus, the real vector space $\mathfrak{M}_{0, {\rm sa}}$ has the symmetric bilinear form $s_\chi$ and the symplectic form $\sigma_\chi$ given by
\begin{gather*}
 s_\chi(x, y):=\operatorname{Re}\sum_{k\in\mathbb{Z}}(\chi(x\gamma_k(y))-\chi(x)\chi(y)),\qquad\!\! \sigma_\chi(x, y):=-\operatorname{Im}\sum_{k\in\mathbb{Z}}(\chi(x\gamma_k(y))-\chi(x)\chi(y))
\end{gather*}
for any $x, y\in\mathfrak{M}_{0, {\rm sa}}$. We denote by $\mathfrak{W}_\chi$ the \emph{Weyl CCR algebra} associated with the symplectic space $(\mathfrak{M}_{0, {\rm sa}}, \sigma_\chi)$, that is, $\mathfrak{W}_\chi$ is the universal $C^*$-algebra generated by $w(x)$ ($x\in\mathfrak{M}_{0, {\rm sa}}$) subject to the relations
\[w(x)^*=w(-x),\qquad w(x)w(y)={\rm e}^{\sigma_\chi(x, y)}w(x+y),\qquad x, y\in\mathfrak{M}_{0, {\rm sa}}.\]
We remark that the kernel $(x, y)\mapsto s_\chi(x, y)-\mathrm{i}\sigma_\chi(x, y)$ on $\mathfrak{M}_{0, {\rm sa}}$ is positive definite since $s_\chi(x, y)-\mathrm{i}\sigma_\chi(x, y)=\lim_{n\to\infty}\chi(F_n(x)F_n(y))$. Thus, we obtain a quasi-free state $\varphi_\chi$ on $\mathfrak{W}_\chi$ such that $\varphi_\chi(w(x))={\rm e}^{-s_\chi(x, x)/2}$ for any $x\in \mathfrak{M}_{0, {\rm sa}}$.

The following is the main theorem in this section and the proof is given in Appendix \ref{app:proof}.
\begin{Theorem}\label{theorem:gaussian}
 For any $k\in\mathbb{N}$ and $x_1, \dots, x_k\in\mathfrak{M}_{0, {\rm sa}}$, we have
 \[\lim_{n\to\infty}\chi\big({\rm e}^{\mathrm{i}F_n(x_1)}\cdots {\rm e}^{\mathrm{i}F_n(x_k)}\big)=\varphi_\chi(w(x_1)\cdots w(x_k)).\]
\end{Theorem}

\begin{Remark}
 This theorem states that the local fluctuations $F_n(x_1),\dots, F_n(x_k)$ converge to Gaussian family in the following sense: Let $(\pi_\chi, \mathfrak{h}_\chi, \xi_\chi)$ be the GNS-triple associated with $\chi$. For any $x\in\mathfrak{M}_{0, {\rm sa}}$, the function $t\in\mathbb{R}\mapsto \varphi_\chi(w(tx))={\rm e}^{-t^2s_\chi(x, x)/2}$ is analytic. Thus, there exists a self-adjoint operator $b(x)$ on $\mathfrak{h}_\chi$ satisfying that $\pi_\chi(w(tx))={\rm e}^{\mathrm{i}tb(x)}$ for any $t\in\mathbb{R}$. Moreover, $w(x)\xi_\chi$ is in the domain of $b(x_1)\cdots b(x_k)$ for any $x, x_1,\dots, x_k\in\mathfrak{M}_{0, {\rm sa}}$. We use the same symbol~$\varphi_\chi$ to denote the vector state given by $\xi_\chi$, that is, we define
 \[\varphi_\chi(b(x_1)\cdots b(x_k)):=\langle b(x_1)\cdots b(x_k)\xi_\chi, \xi_\chi\rangle.\]
 By Theorem \ref{theorem:gaussian}, we have $\lim_{n\to\infty}\chi\big({\rm e}^{\mathrm{i}F_n(x_1)}\cdots {\rm e}^{\mathrm{i}F_n(x_k)}\big)=\varphi_\chi\big({\rm e}^{\mathrm{i}b(x_1)}\cdots {\rm e}^{\mathrm{i}b(x_k)}\big)$, i.e., the characteristic function of $F_n(x_1),\dots, F_n(x_k)$ converges to the characteristic function of $b(x_1), \dots, b(x_k)$. On the other hand, the following Wick formula holds true (see \cite[Proposition 3.8]{Petz:book}):
 \[\varphi_\chi(b(x_1)\cdots b(x_k))=\begin{cases}\displaystyle \sum_\pi\prod_{m=1}^{k/2}s_\chi(x_{i_m}, x_{j_m})-\mathrm{i}\sigma_\chi(x_{i_m}, x_{j_m}),& k\ \text{is even},\\ 0,&k \ \text{is odd},\end{cases}\]
 where the summation runs over all pair partitions $\pi=\{\{i_1>j_1\},\dots, \{i_{k/2}>j_{k/2}\}\}$ of $\{1, 2, \dots, k\}$. In particular, for any $x, y\in \mathfrak{M}_{0, {\rm sa}}$ the covariance of $b(x)$ and $b(y)$ is given as
 \[\varphi_\chi(b(x)b(y))=s_\chi(x, y)-\mathrm{i}\sigma_\chi(x, y)=\sum_{k\in\mathbb{Z}}(\chi(x\gamma_k(y))-\chi(x)\chi(y)).\]
 See \cite{AB, D, GVV89,GVV90, GVV91, GV89, Matsui} for this type of non-commutative central limit theorems.
\end{Remark}

We now apply Theorem \ref{theorem:gaussian} to multiplicative characters of ${\rm U}(\infty)$. We recall that $\mathfrak{M}(U)$ is the $C^*$-inductive limit of the group $W^*$-algebras $W^*({\rm U}(I))$. Let $\mathfrak{M}(U)_0:=\bigcup_{I\in\mathcal{I}}W^*({\rm U}(I))$ and~$\mathfrak{M}(U)_{0, {\rm sa}}$ the space of all self-adjoint elements in $\mathfrak{M}(U)_0$. For any $\omega\in \Omega$, we denote by $\chi_\omega$ the locally normal tracial state on $\mathfrak{M}(U)$ satisfying equation \eqref{eq:char_omega} with $q=1$. By Proposition~\ref{prop:main3}, $\chi_\omega$ is multiplicative on $\mathfrak{M}(U)$.

By equation \eqref{eq:char_omega} (see also \cite[Lemma 7.1]{Sato4}), only if $q=1$, then $\chi_\omega$ is $\gamma$-invariant, and hence the corresponding symmetric bilinear form $s_{\chi_\omega}$ and the symplectic form $\sigma_{\chi_\omega}$ are well defined. Let $\mathfrak{W}_{\chi_\omega}$ denote the Weyl CCR algebra associated with $(\mathfrak{M}(U)_{0, {\rm sa}}, \sigma_{\chi_\omega})$ and $\varphi_{\chi_\omega}$ the quasi-free state on $\mathfrak{W}_{\chi_\omega}$ associated with $s_{\chi_\omega}$.

\begin{Corollary}\label{cor:cha}
 For any $k\in\mathbb{N}$ and $x_1, \dots, x_k\in\mathfrak{M}(U)_{0, {\rm sa}}$, we have
 \[\lim_{n\to\infty}\chi_\omega\big({\rm e}^{\mathrm{i}F_n(x_1)}\cdots {\rm e}^{\mathrm{i}F_n(x_k)}\big)=\varphi_{\chi_\omega}(w(x_1)\cdots w(x_k)).\]
\end{Corollary}

By this corollary, the extreme (i.e., multiplicative) characters of ${\rm U}(\infty)$ produce Gaussian fluctuation limits on the group $W^*$-algebras $W^*({\rm U}(I))$. Moreover, since $\chi_\omega$ is tracial, we have~$s_{\chi_\omega}\equiv0$, and hence the Weyl CCR algebra $\mathfrak{W}_{\chi_\omega}$ is commutative. Namely, we obtain a commutative Gaussian family as limits of $F_n(x_1),\dots, F_n(x_k)$ for any $x_1, \dots, x_k\in\mathfrak{M}(U)_{0, {\rm sa}}$. However, to the best of the author's knowledge, a probabilistic description of them is unclear.

\appendix

\section{Proof of Theorem \ref{theorem:gaussian}}\label{app:proof}
Here we show Theorem \ref{theorem:gaussian}. The strategy is essentially the same as \cite{GVV90, GV89}, and hence the proof might be clear for the experts. However, we give this appendix for the reader's convenience. We use the same notations in the previous section.

For any $n=1, 2,\dots$, we set
\[p_n:=\big\lfloor |I_n|^{1/4}\log |I_n|\big\rfloor, \qquad q_n:=\left\lfloor\frac{|I_n|^{1/2}}{p_n}\right\rfloor, \qquad m_n=\left\lfloor\frac{|I_n|}{p_n+q_n}\right\rfloor\]
and take a partition $I_n=J^{(n)}_1\sqcup L^{(n)}_1\sqcup J^{(n)}_2\sqcup\cdots\sqcup L^{(n)}_{m_n}\sqcup L^{(n)}_{m_n+1}$ such that $J^{(n)}_l$, $L^{(n)}_l$ are intervals with $\big|J^{(n)}_l\big|=p_n$ and $|L^{(n)}_l|=q_n$ for every $l=1,\dots, m_n$ and $J^{(n)}_1, L^{(n)}_1,\dots, L^{(n)}_{m_n}, L^{(n)}_{m_n+1}$ lie in order from left to right (see Figure \ref{fig:partition}). We remark that $\big|L^{(n)}_{m_n+1}\big|<p_n+q_n$. For $x\in\mathfrak{M}$, we~define
\[s^J_n(x):=\frac{1}{\sqrt{|I_n|}}\sum_{k\in J_n}(\gamma_k(x)-\chi(x)), \qquad s^L_n(x):=\frac{1}{\sqrt{|I_n|}}\sum_{k\in L_n}(\gamma_k(x)-\chi(x)),\]
where $J_n:=\bigsqcup_{l=1}^{m_n}J^{(n)}_l$ and $L_n:=\bigsqcup_{l=1}^{m_n+1}L^{(n)}_l$. We remark that $F_n(x)=s^J_n(x)+s^L_n(x)$.

 \begin{figure}[htbp]\centering
 \begin{tikzpicture}
 \draw (0bp, 0bp)--node[above]{$I_n$}(300bp, 0bp);
 \draw(25bp, 0bp)node[below]{$J^{(n)}_1$};
 \draw(65bp, 0bp) node[below]{$L^{(n)}_1$};
 \draw(105bp, 0bp) node[below]{$J^{(n)}_2$};
 \draw(145bp, 0bp) node[below]{$L^{(n)}_2$};
 \draw(265bp, 0bp) node[below]{$L^{(n)}_{m_n}$};
 \draw(290bp, 0bp) node[below]{$L^{(n)}_{m_n+1}$};
 \draw (0bp, -3bp)--(0bp, 3bp);
 \draw (50bp, -3bp)--(50bp, 3bp);
 \draw (80bp, -3bp)--(80bp, 3bp);
 \draw (130bp, -3bp)--(130bp, 3bp);
 \draw (160bp, -3bp)--(160bp, 3bp);
 \draw (250bp, -3bp)--(250bp, 3bp);
 \draw (280bp, -3bp)--(280bp, 3bp);
 \draw (300bp, -3bp)--(300bp, 3bp);
 \end{tikzpicture}
 \caption{Partition of $I_n$.}\label{fig:partition}
 \end{figure}

\begin{Lemma}\label{lem:JL}
 Let $x, y, z\in\mathfrak{M}_0$. The following hold true:
 \begin{itemize}\itemsep=0pt
 \item[$(1)$] $ \lim_{n\to\infty}\big\|\big[s^J_n(x), s^L_n(y)\big]\big\|\!=\!0$, $\lim_{n\to\infty}\big\|\big[s^J_n(x), s^J_n(y)\big]\big\|\!<\!\infty$, $\lim_{n\to\infty}\big\|\big[s^L_n(x), s^L_n(y)\big]\big\|\!=\!0$.
 \item[$(2)$] $\lim_{n\to\infty}\big\|\big[s_n^L(x), [s^J_n(y), s^J_n(z)]\big]\big\|\!=\!0$.
 \end{itemize}
\end{Lemma}
\begin{proof}
 For the claim (1), we show only the first case, and the other statements can be proved by similar ways. By definition, $p_n/|I_n|$, $q_n/|I_n|$, and $q_nm_n/|I_n|$ converge to 0 as $n\to\infty$. Thus, we have
 \begin{align*}
 \big\|\big[s^J_n(x), s^L_n(y)\big]\big\|
 & \leq \frac{1}{\sqrt{|I_n|}}\sum_{k\in J_n}\big\|\big[\gamma_k(x), s^L_n(y)\big]\big\|
 \leq \frac{1}{|I_n|}\sum_{l\in L_n}\sum_{k\in\mathbb{Z}}\|[\gamma_k(x), y]\| \\
 & \leq \frac{q_nm_n+p_n+q_n}{|I_n|}\sum_{k\in\mathbb{Z}}\|[\gamma_k(x), y]\| \to0\qquad \text{as }n\to\infty,
 \end{align*}
 where we remark that $\sum_{k\in\mathbb{Z}}\|[\gamma_k(x), y]\|<\infty$ since $x, y\in\mathfrak{M}_0$.

 We show the claim (2). By the Jacobi identity, it suffices to show that $\big\| \big[s^J_n(y), \!\big[s^J_n(z), s^L_n(x)\big]\big]\big\|$ converges to 0 as $n\to\infty$, and we have
 \begin{align*}
 \big\| \big[s^J_n(y), \big[s^J_n(z), s^L_n(x)\big]\big]\big\|
 & \leq \frac{1}{|I_n|}\sum_{k_1, k_2\in J_n}\big\|\big[\gamma_{k_1}(y), \big[\gamma_{k_2}(z), s^L_n(x)\big]\big]\big\| \\
 & \leq \frac{q_nm_n+p_n+q_n}{|I_n|^{3/2}}\sum_{k_1, k_2\in \mathbb{Z}}\|[\gamma_{k_1}(y), [\gamma_{k_2}(z), x]]\| \\
 & \to0 \qquad \text{as }n\to\infty,
 \end{align*}
 where we remark that $\sum_{k_1, k_2\in\mathbb{Z}}\|[\gamma_{k_1}(x), [\gamma_{k_2}(y), z]]<\infty$ since $x, y, z\in \mathfrak{M}_0$.
\end{proof}

We obtain the following Gaussian fluctuation limit of single operator.
\begin{Proposition}\label{prop:k=1}
 For any $x\in\mathfrak{M}_{0, {\rm sa}}$, we have $\lim_{n\to\infty}\chi\big({\rm e}^{\mathrm{i}F_n(x)}\big)=\varphi_\chi(w(x))$.
\end{Proposition}
\begin{proof}
 We may assume that $\chi(x)=0$. Using the formulas ${\rm e}^{\mathrm{i}(a+b)}-{\rm e}^{\mathrm{i}a}=\int_0^1{\rm e}^{\mathrm{i}(1-t)a}(\mathrm{i}b){\rm e}^{\mathrm{i}t(a+b)}{\rm d}t$ and $\|[{\rm e}^{\mathrm{i}a}, b]\|\leq \|[a, b]\|$ for bounded self-adjoint operators $a$, $b$ (see \cite[Proof of Theorem 4.1]{GVV89}), we have
 \begin{align*}
 \big|\chi\big({\rm e}^{\mathrm{i}F_n(x)}\big)-\chi\big({\rm e}^{\mathrm{i}s^J_n(x)}\big)\big|
 & \leq\int_0^1\big|\chi\big({\rm e}^{\mathrm{i}(1-t)s^J_n(x)}s^L_n(x){\rm e}^{\mathrm{i}tF_n(x)}\big)\big|{\rm d}t \\
 & \leq \int_0^1\big|\chi\big(\big[{\rm e}^{\mathrm{i}(1-t)s^J_n(x)}, s^L_n(x)\big]{\rm e}^{\mathrm{i}tF_n(x)}\big)\big|{\rm d}t\\
 &\quad+\int_0^1\big|\chi\big(s^L_n(x){\rm e}^{\mathrm{i}(1-t)s^J_n(x)}{\rm e}^{\mathrm{i}tF_n(x)}\big)\big|{\rm d}t \\
 & \leq\frac{1}{2}\big\|\big[s^J_n(x), s^L_n(x)\big]\big\|+\big|\chi\big(s^L_n(x)^2\big)\big|^{1/2},
 \end{align*}
 where we remark that $\big|\chi\big(s^L_n(x){\rm e}^{\mathrm{i}(1-t)s^J_n(x)}{\rm e}^{\mathrm{i}tF_n(x)}\big)\big|\leq \big|\chi\big(s^L_n(x)^2\big)\big|^{1/2}$ by the Cauchy--Schwarz inequality. By Lemma \ref{lem:JL}, $\big\|\big[s^J_n(x), s^L_n(x)\big]\big\|$ converges to $0$ as $n\to\infty$. We also have
 \begin{align*}
 \big|\chi\big(s^L_n(x)^2\big)\big|
 & \leq\frac{1}{|I_n|}\sum_{k_1, k_2\in L_n}|\chi(x\gamma_{k_2-k_1}(x))| \leq\frac{m_nq_n+p_n+q_n}{|I_n|}\sum_{k\in\mathbb{Z}}|\chi(x\gamma_k(x))| \to0
 \end{align*}
 as $n\to\infty$, where $\sum_{k\in\mathbb{Z}}|\chi(x\gamma_k(x))|<\infty$ since $\chi(x)=0$ and $\chi$ is multiplicative. Thus, we have
 \[\lim_{n\to\infty}\big|\chi\big({\rm e}^{\mathrm{i}F_n(x)}\big)-\chi\big({\rm e}^{\mathrm{i}s^J_n(x)}\big)\big|=0.\]
 Therefore, it suffices to show that $\lim_{n\to\infty}\chi({\rm e}^{\mathrm{i}s^J_n(x)})={\rm e}^{-s_\chi(x, x)/2}$. Since $x\in\mathfrak{M}_0$, for sufficiently large $n$, the intervals $L^{(n)}_l$ are large enough and $\{F_{J^{(n)}_l}(x)\}_{l=1}^{m_n}$ commute mutually. Since $\chi$ is $\gamma$-invariant and multiplicative, we have
 \[\chi\big({\rm e}^{\mathrm{i}s^J_n(x)}\big)=\chi\Big({\rm e}^{\mathrm{i}\sqrt{\frac{p_n}{|I_n|}}F_{J^{(n)}_1}(x)}\cdots {\rm e}^{\mathrm{i}\sqrt{\frac{p_n}{|I_n|}}F_{J^{(n)}_{m_n}}(x)}\Big)=\chi\Big({\rm e}^{\mathrm{i}\sqrt{\frac{p_n}{|I_n|}}F_{J^{(n)}_1}(x)}\Big)^{m_n}.\]
 By the Taylor expansion theorem, there exists $t\in[0, 1]$ such that
 \begin{align*}
 \chi\Big({\rm e}^{\mathrm{i}\sqrt{\frac{p_n}{|I_n|}}F_{J^{(n)}_1}(x)}\Big)
 & =1-\frac{p_n}{2|I_n|}\chi\big(F_{J^{(n)}_1}(x)^2\big)+\frac{1}{3!}\sqrt{\frac{p_n}{|I_n|}}^3\chi\Big(F_{J^{(n)}_1}(x)^3{\rm e}^{\mathrm{i}t\sqrt{\frac{p_n}{|I_n|}}F_{J^{(n)}_1}(x)}\Big) \\
 & =1-\frac{p_n}{2|I_n|}\big(\chi\big(F_{J^{(n)}_1}(x)^2\big)+O(\sqrt{p_n^3/|I_n|})\big),
 \end{align*}
 where we remark that
 \begin{align*}
 \Big|\chi\Big(F_{J^{(n)}_1}(x)^3{\rm e}^{\mathrm{i}t\sqrt{\frac{p_n}{|I_n|}}F_{J^{(n)}_1}(x)}\Big)\Big|
 & \leq \sqrt{\chi\big(F_{J^{(n)}_1}(x)^2\big)}\sqrt{\chi\big(F_{J^{(n)}_1}(x)^4\big)} \leq p_n\|x\|^2\sqrt{\chi\big(F_{J^{(n)}_1}(x)^2\big)},
 \end{align*}
 and $\lim_{n\to \infty}\chi\big(F_{J^{(n)}_1}(x)^2\big)=s_\chi(x, x)$. Therefore, we have
 \[\chi\Big({\rm e}^{\mathrm{i}\sqrt{\frac{p_n}{|I_n|}}F_{J^{(n)}_1}(x)}\Big)^{m_n}=\Big(1-\frac{p_n}{2|I_n|}\Big(\chi\big(F_{J^{(n)}_1}(x)^2\big)+O\big(\sqrt{p_n^3/|I_n|}\big)\Big)\Big)^{m_n}\to {\rm e}^{-s_\chi(x, x)/2}\]
 as $n\to\infty$.
\end{proof}

For any $x, y\in\mathfrak{M}$, we define $L(x, y):={\rm e}^{\mathrm{i}x}{\rm e}^{\mathrm{i}y}-{\rm e}^{\mathrm{i}(x+y)}{\rm e}^{-\frac{1}{2}[x, y]}$.
\begin{Lemma}\label{lem:L}
 For any $x, y\in\mathfrak{M}_{0,{\rm sa}}$, we have $\lim_{n\to\infty}\|L(F_n(x), F_n(y))\|=0$.
\end{Lemma}
\begin{proof}
 It suffices to show that the following three terms
 \begin{gather*}
 A_n:=\big\|{\rm e}^{\mathrm{i}F_n(x)}{\rm e}^{\mathrm{i}F_n(y)}-{\rm e}^{\mathrm{i}s^J_n(x)}{\rm e}^{\mathrm{i}s^J_n(y)}{\rm e}^{\mathrm{i}s^L_n(x)}{\rm e}^{\mathrm{i}s^L_n(y)}\big\|,\\
 B_n:=\big\|{\rm e}^{\mathrm{i}s^J_n(x)}{\rm e}^{\mathrm{i}s^J_n(y)}{\rm e}^{\mathrm{i}s^L_n(x)}{\rm e}^{\mathrm{i}s^L_n(y)}-{\rm e}^{\mathrm{i}s^J_n(x+y)}{\rm e}^{-\frac{1}{2}[s^J_n(x), s^J_n(y)]}{\rm e}^{\mathrm{i}s^L_n(x+y)}{\rm e}^{-\frac{1}{2}[s^L_n(x), s^L_n(y)]}\big\|,\\
 C_n:=\big\|{\rm e}^{\mathrm{i}s^J_n(x+y)}{\rm e}^{-\frac{1}{2}[s^J_n(x), s^J_n(y)]}{\rm e}^{\mathrm{i}s^L_n(x+y)}{\rm e}^{-\frac{1}{2}[s^L_n(x), s^L_n(y)]}-{\rm e}^{\mathrm{i}F_n(x+y)}{\rm e}^{-\frac{1}{2}[F_n(x), F_n(y)]}\big\|
 \end{gather*}
 converge to $0$ as $n\to\infty$. To show them, we use the following inequalities (see, e.g., \cite[proof of Lemma 2.2]{Matsui02}, \cite[proof of Theorem 4.1]{GVV89}):
 \begin{gather*}
 \big\|{\rm e}^{\mathrm{i(a+b)}}-{\rm e}^{\mathrm{i}a}{\rm e}^{\mathrm{i}b}\big\|\leq\frac{1}{2}\|[a, b]\|, \qquad
 \big\|\big[{\rm e}^{\mathrm{i}a}, {\rm e}^{\mathrm{i}b}\big]\big\|\leq \big\|\big[a, {\rm e}^{\mathrm{i}b}\big]\big\|\leq \|[a, b]\|
 \end{gather*}
 for any self-adjoint elements $a$, $b$ of a $C^*$-algebra. By Lemma \ref{lem:JL}, we have
 \begin{align*}
 A_n & \leq \big\|{\rm e}^{\mathrm{i}F_n(x)}-{\rm e}^{\mathrm{i}s^J_n(x)}{\rm e}^{\mathrm{i}s^L_n(x)}\big\|+\big\|{\rm e}^{\mathrm{i}F_n(y)}-{\rm e}^{\mathrm{i}s^J_n(y)}{\rm e}^{\mathrm{i}s^L_n(y)}\big\|+\big\|[{\rm e}^{\mathrm{i}s^L_n(x)}, {\rm e}^{\mathrm{i}s^J_n(y)}]\big\| \\
 & \leq \frac{1}{2}\big\|\big[s^J_n(x), s^L_n(x)\big]\big\|+\frac{1}{2}\big\|\big[s^J_n(y), s^L_n(y)\big]\big\|+\big\|\big[s^L_n(x), s^J_n(y)\big]\big\| \to 0\qquad \text{as }n\to\infty,
 \end{align*}
 and
 \begin{align*}
 C_n & \leq \big\|\big[{\rm e}^{-\frac{1}{2}[s^J_n(x), s^J_n(y)]}, {\rm e}^{\mathrm{i}s^L_n(x+y)}\big]\big\|{\rm e}^{\frac{1}{2}\|[s^L_n(x), s^L_n(y)]\|} \\&\quad+\big\|{\rm e}^{-\frac{1}{2}[s^J_n(x), s^J_n(y)]}{\rm e}^{-\frac{1}{2}[s^L_n(x), s^L_n(y)]}-{\rm e}^{-\frac{1}{2}[s^J_n(x), s^J_n(y)]-\frac{1}{2}[s^L_n(x), s^L_n(y)]}\big\|\\
 &\quad +\big\|{\rm e}^{-\frac{1}{2}[s^J_n(x), s^J_n(y)]-\frac{1}{2}[s^L_n(x), s^L_n(y)]}-{\rm e}^{-\frac{1}{2}[F_n(x), F_n(y)]}\big\|\\
 &\quad+ \big\|{\rm e}^{\mathrm{i}s^J_n(x+y)}{\rm e}^{\mathrm{i}s^L_n(x+y)}-{\rm e}^{\mathrm{i}F_n(x+y)}\big\|{\rm e}^{\frac{1}{2}\|[F_n(x), F_n(y)]\|}\\
 & \leq \frac{1}{2}\big\|\big[s^L_n(x+y), \big[s^J_n(x), s^J_n(y)\big]\big]\big\|{\rm e}^{\frac{1}{2}\|[s^L_n(x), s^L_n(y)]\|} \\
 &\quad+\frac{1}{8}\big\|\big[\big[s^J_n(x), s^J_n(y)\big], \big[s^L_n(x), s^L_n(y)\big]\big]\big\|\\
 &\quad +\frac{1}{8}\big\|\big[\big[s^J_n(x), s^J_n(y)\big]+\big[s^L_n(x), s^L_n(y)\big], \big[s^J_n(x), s^L_n(y)\big]+\big[s^L_n(x), s^J_n(y)\big]\big]\big\|\\
 &\quad +\frac{1}{2}\big\|\big[s^J_n(x+y), s^L_n(x+y)\big]\big\|{\rm e}^{\frac{1}{2}\|[F_n(x), F_n(y)]\big\|}
 \to0 \qquad\text{as }n\to\infty,
 \end{align*}
 where we remark that $\|[F_n(x), F_n(y)]\|\leq \sum_{k\in\mathbb{Z}}\|[\gamma_k(x), y]\|<\infty$. To show $\lim_{n\to\infty}B_n=0$, it suffices to show that $\big\|L\big(s^J_n(x), s^J_n(y)\big)\big\|$, $\big\|L\big(s^L_n(x), s^L_n(y)\big)\big\|$ converge to $0$ as $n\to\infty$ since
 \[B_n \leq \big\|L\big(s^J_n(x), s^J_n(y)\big)\big\|+{\rm e}^{\frac{1}{2}\|[s^J_n(x), s^J_n(y)]\|}\big\|L\big(s^L_n(x), s^L_n(y)\big)\big\|.\]
 For large $n$, we have $[F_{J^{(n)}_l}(x), F_{J^{(n)}_m}(x)]=[F_{J^{(n)}_l}(y), F_{J^{(n)}_m}(y)]=[F_{J^{(n)}_l}(x), F_{J^{(n)}_m}(y)]=0$ if $l\neq m$. Thus, applying \cite[Lemma 5.3]{Petz:book} repeatedly, we have
 \[\big\|L\big(s^J_n(x), s^J_n(y)\big)\big\|\leq \left\|L\left(\sqrt{\frac{p_n}{|I_n|}}F_{J^{(n)}_1}(x), \sqrt{\frac{p_n}{|I_n|}}F_{J^{(n)}_1}(y)\right)\right\|\frac{1-{\rm e}^{\frac{m_n}{2}\|\sqrt{\frac{p_n}{|I_n|}}F_{J^{(n)}_1}\|}}{1-{\rm e}^{\frac{1}{2}\|\sqrt{\frac{p_n}{|I_n|}}F_{J^{(n)}_1}\|}}.\]
 Moreover, we have \smash{$\sqrt{p_n/|I_n|}\|F_{J^{(n)}_1}(x)\|\leq \sqrt{p_n^2/|I_n|}\|x\|$} and $p_n^2/|I_n|\to0$ as $n$ goes to infinity. Thus, by \cite[Lemma 5.2]{Petz:book}, we conclude that $\lim_{n\to\infty}\big\|L\big(s^J_n(x), s^J_n(y)\big)\big\|=0$. Similarly, we have $\lim_{n\to\infty}\big\|L\big(s^L_n(x), s^L_n(y)\big)\big\|=0$.
\end{proof}

\begin{Lemma}\label{lem:bdd_seq}
 Let $y\in \mathfrak{M}_0$ and $(x_n)_{n=1}^\infty$ a bounded sequence in $\mathfrak{M}$ such that $\lim_{n\to\infty}\chi(x_n)$ exists. Then, we have
 \[\lim_{n\to\infty}\chi\Big(x_n{\rm e}^{\frac{1}{|I_n|}\sum_{k\in I_n}\gamma_k(y)}\Big)=\lim_{n\to\infty}\chi(x_n){\rm e}^{\chi(y)}.\]
\end{Lemma}
\begin{proof}
 We may assume that $\chi(y)=0$. Let $(\pi_\chi, \mathfrak{h}_\chi, \xi_\chi)$ denote the GNS-triple associated with~$\chi$ and $V$ the unitary operator on $\mathfrak{h}_\chi$ given by $V\pi_\chi(x)\xi_\chi=\pi_\chi(\gamma_1(x))\xi_\chi$ for any $x\in\mathfrak{M}$. By the von Neumann ergodic theorem, $\frac{1}{|I_n|}\sum_{k\in I_n}V^k$ converges strongly to the projection $P_V$ onto the subspace of $V$-invariant vectors. Moreover, for any $x, y\in\mathfrak{M}_0$ we have
 \[\langle \pi_\chi(y)\xi_\chi, P_V\pi_\chi(x)\xi_\chi\rangle=\lim_{n\to\infty}\frac{1}{|I_n|}\sum_{k\in I_n}\chi(\gamma_k(x^*)y)=\chi(x^*)\chi(y)=\langle\pi_\chi(y)\xi_\chi, \chi(x)\xi_\chi\rangle.\]
 Since $\xi_\chi$ is cyclic for $\pi_\chi(\mathfrak{M}_0)$, we have $P_V\pi_\chi(x)\xi_\chi=\chi(x)\xi_\chi$ for any $x\in\mathfrak{M}_0$. Thus, we have
 \begin{align*}
 \Big|\chi\Big(x_n{\rm e}^{\frac{1}{|I_n|}{}\sum_{k\in I_n}\gamma_k(y)}\Big)-\chi(x_n)\Big| &=\Big|\Big\langle \pi_\chi\Big({\rm e}^{\frac{1}{|I_n|}\sum_{k\in I_n}\gamma_k(y)}-1\Big)\xi_\chi, \pi_\chi(x_n^*)\xi_\chi\Big\rangle\Big| \\
 & \leq\bigg\|\frac{1}{|I_n|}\sum_{k\in I_n}\pi_\chi(\gamma_k(y))\xi_\chi\bigg\| \\
 &\quad\times\bigg\|\sum_{m=0}^\infty\frac{1}{(m+1)!}\bigg(\frac{1}{|I_n|}\sum_{k\in I_n}\pi_\chi(\gamma_k(y))\bigg)^m \pi_\chi(x_n^*)\xi_\chi\bigg\| \\
 & \leq\bigg\|\frac{1}{|I_n|}\sum_{k\in I_n}V^k\pi_\chi(y)\xi_\chi\bigg\|{\rm e}^{\|y\|}\sup_n\|x_n\| \\
 & \to0 \qquad \text{as }n\to\infty.\tag*{\qed}
 \end{align*} \renewcommand{\qed}{}
\end{proof}

Now we prove Theorem \ref{theorem:gaussian}.
\begin{proof}[Proof of Theorem \ref{theorem:gaussian}]
 We may assume that $\chi(x_i)=0$ ($i=1,\dots, k$). If $k=1$, the assertion follows from Proposition \ref{prop:k=1}. For $1,\dots, k-1$, we assume that the statement holds true. Thus, we have
 \[\lim_{n\to\infty}\chi\big({\rm e}^{\mathrm{i}F_n(x_1)}\cdots {\rm e}^{\mathrm{i}F_n(x_{k-2})}{\rm e}^{\mathrm{i}F_n(x_{k-1}+x_k)}\big)=\varphi_\chi(w(x_1)\cdots w(x_{k-2})w(x_{k-1}+x_k)).\]
 By Lemma \ref{lem:L}, we have
 \[\lim_{n\to\infty}\big|\chi\big({\rm e}^{\mathrm{i}F_n(x_1)}\cdots {\rm e}^{\mathrm{i}F_n(x_k)}\big)-\chi\big({\rm e}^{\mathrm{i}F_n(x_1)}\cdots {\rm e}^{\mathrm{i}F_n(x_{k-2})}{\rm e}^{\mathrm{i}F_n(x_{k-1}+x_k)}{\rm e}^{-[F_n(x_{k-1}), F_n(x_k)]/2}\big)\big|=0,\]
 and $[F_n(x_{k-1}), F_n(x_k)]=\frac{1}{|I_n|}\sum_{j\in I_n}\gamma_j(\sum_{l\in\mathbb{Z}}[x_{k-1}, \gamma_l(x_k)])$ for sufficiently large $n$. Therefore, by Lemma \ref{lem:bdd_seq}, we have
 \begin{align*}
 & \lim_{n\to\infty}\chi\big({\rm e}^{\mathrm{i}F_n(x_1)}\cdots {\rm e}^{\mathrm{i}F_n(x_k)}\big), \\
 & \lim_{n\to\infty}\chi\big({\rm e}^{\mathrm{i}F_n(x_1)}\cdots {\rm e}^{\mathrm{i}F_n(x_{k-2})}{\rm e}^{\mathrm{i}F_n(x_{k-1}+x_k)}{\rm e}^{-[F_n(x_{k-1}), F_n(x_k)]/2}\big) \\
 & \qquad=\varphi_\chi(w(x_1)\cdots w(x_{k-2})w(x_{k-1}+x_k)){\rm e}^{\mathrm{i}\sigma_\chi(x_{k-1}, x_k)} \\
 &\qquad =\varphi_\chi(w(x_1)\cdots w(x_{k-1})w(x_k)).\tag*{\qed}
 \end{align*} \renewcommand{\qed}{}
\end{proof}

\subsection*{Acknowledgements}
The author gratefully acknowledges comments from Professor Makoto Katori at the early stage of this work. The author is grateful to Professor Yoshimichi Ueda for the useful discussion and his comments on the draft of this paper. Finally, the author gratefully thanks the referees for their kind reading and comments.
This work was supported by JSPS Research Fellowship for Young Scientists PD (KAKENHI Grand Number 22J00573).

\pdfbookmark[1]{References}{ref}
\LastPageEnding

\end{document}